\documentclass[twoside,a4paper,10pt]{amsart}

\usepackage{fullpage}
\usepackage{amscd}
\usepackage{amsmath}
\usepackage{amssymb}
\usepackage{amsthm}
\usepackage{color}
\usepackage{graphicx}
\usepackage{hyperref}
\usepackage{mathrsfs}
\usepackage{mathtools}
\usepackage{pict2e}
\usepackage{stackrel}
\usepackage[T1]{fontenc}
\usepackage{tikz}
\usepackage[utf8]{inputenc}
\usepackage{wasysym}
\usepackage[all]{xy}
\usepackage{xcolor}

\hypersetup{
    colorlinks,
    linkcolor={red!50!black},
    citecolor={blue!50!black},
    urlcolor={blue!80!black}
}

\usepackage{palatino}
\allowdisplaybreaks

\newtheorem{lemma}{Lemma}[section]
\newtheorem{remark}[lemma]{Remark}

\newtheorem{theorem}[lemma]{Theorem}
\newtheorem{corollary}[lemma]{Corollary}
\newtheorem{definition}[lemma]{Definition}
\newtheorem{proposition}[lemma]{Proposition}

\newtheorem*{mainthm}{Theorem \ref{thm:mainThmDual} and \ref{thm:mainThm}}
\newtheorem*{thmbifunctor}{Theorem \ref{thm:bifunctor}}
\newtheorem*{thmtransport}{Theorem \ref{thm:twoLinftyStrAreEqual}}
\newtheorem*{thmMC}{Corollaries \ref{cor:MCofHom} and \ref{cor:MCel}}

\author{Daniel Robert-Nicoud}
\date{}
\title{Deformation theory with homotopy algebra structures on tensor products}
\address{Laboratoire Analyse, G\'eom\'etrie et Applications, Universit\'e Paris 13, Sorbonne Paris Cit\'e, 99 Avenue Jean Baptiste Cl\'ement, 93430 Villetaneuse, France}
\email{robert-nicoud@math.univ-paris13.fr}

\newcommand{\A}{\ensuremath{\mathscr{A}}}

\newcommand{\antishriek}{\text{\raisebox{\depth}{\textexclamdown}}}
\newcommand{\as}{\ensuremath{\mathit{As}}}
\newcommand{\ass}{\ensuremath{\mathit{Ass}}}
\renewcommand{\bar}{\ensuremath{\mathrm{B}}}
\newcommand{\bmanin}{\,\mbox{\tikz\draw[black,fill=black] (0,0) circle (.7ex);}\,}
\newcommand{\brt}{\ensuremath{\mathrm{BRT}}}
\newcommand{\C}{\ensuremath{\mathscr{C}}}

\newcommand{\ch}{\ensuremath{\textsf{Chain}}}

\newcommand{\com}{\ensuremath{\mathit{Com}}}

\newcommand{\End}{\ensuremath{\mathrm{End}}}
\newcommand{\g}{\ensuremath{\mathfrak{g}}}
\newcommand{\genmanin}{\ensuremath{\mathsf{M}}}
\newcommand{\genmaninDual}{\rotatebox[origin=c]{180}{$\mathsf{M}$}}
\newcommand{\genmaninns}{\ensuremath{\mathsf{M}^{\mathrm{ns}}}}
\newcommand{\genmaninnsDual}{\ensuremath{\rotatebox[origin=c]{180}{$\mathsf{M}$}^{\mathrm{ns}}}}

\newcommand{\id}{\ensuremath{\mathrm{id}}}
\newcommand{\lie}{\ensuremath{\mathit{Lie}}}

\newcommand{\manin}{\ensuremath{\mathsf{m}}}
\newcommand{\maninns}{\ensuremath{\mathsf{m}^{\mathrm{ns}}}}
\newcommand{\mc}{\ensuremath{\mathfrak{mc}}}
\newcommand{\MC}{\ensuremath{\mathrm{MC}}}

\renewcommand{\k}{\ensuremath{\mathbb{K}}}
\renewcommand{\L}{\ensuremath{\mathscr{L}}}
\renewcommand{\P}{\ensuremath{\mathscr{P}}}

\newcommand{\Q}{\ensuremath{\mathscr{Q}}}
\newcommand{\R}{\ensuremath{\mathscr{R}}}

\renewcommand{\S}{\ensuremath{\mathbb{S}}}
\newcommand{\Sh}{\ensuremath{\text{Sh}}}
\newcommand{\shriek}{\text{\raisebox{\depth}{!}}}
\newcommand{\susp}{\ensuremath{\mathscr{S}}}

\newcommand{\T}{\ensuremath{\mathcal{T}}}
\newcommand{\Tw}{\ensuremath{\mathrm{Tw}}}
\newcommand{\vdw}{\ensuremath{\mathrm{VdL}}}
\newcommand{\wmanin}{\,\mbox{\tikz\draw[black] (0,0) circle (.7ex);}\,}

\makeatletter
\newcommand{\adjunction}{\@ifstar\named@adjunction\normal@adjunction}
\newcommand{\normal@adjunction}[4]{%
  #1\colon #2%
  \mathrel{\vcenter{%
    \offinterlineskip\m@th
    \ialign{%
      \hfil$##$\hfil\cr
      \longrightharpoonup\cr
      \noalign{\kern-.3ex}
      \smallbot\cr
      \longleftharpoondown\cr
    }%
  }}%
  #3 \noloc #4%
}
\newcommand{\named@adjunction}[4]{%
  #2%
  \mathrel{\vcenter{%
    \offinterlineskip\m@th
    \ialign{%
      \hfil$##$\hfil\cr
      \scriptstyle#1\cr
      \noalign{\kern.1ex}
      \longrightharpoonup\cr
      \noalign{\kern-.3ex}
      \smallbot\cr
      \longleftharpoondown\cr
      \scriptstyle#4\cr
    }%
  }}%
  #3%
}
\newcommand{\longrightharpoonup}{\relbar\joinrel\rightharpoonup}
\newcommand{\longleftharpoondown}{\leftharpoondown\joinrel\relbar}
\newcommand\noloc{%
  \nobreak
  \mspace{6mu plus 1mu}
  {:}
  \nonscript\mkern-\thinmuskip
  \mathpunct{}
  \mspace{2mu}
}
\newcommand{\smallbot}{%
  \begingroup\setlength\unitlength{.15em}%
  \begin{picture}(1,1)
  \roundcap
  \polyline(0,0)(1,0)
  \polyline(0.5,0)(0.5,1)
  \end{picture}%
  \endgroup
}
\makeatother

\subjclass[2010]{Primary 18D50; Secondary 08C05, 18G55.}

\keywords{Operads, homotopy Lie algebras, homotopy associative algebras, infinity-morphisms, Maurer--Cartan elements.}

\thanks{The author was supported by grants from R\'egion Ile-de-France, and the grant ANR-14-CE25-0008-01 project SAT}

\begin{document}

\begin{abstract}
	In order to solve two problems in deformation theory, we establish natural structures of homotopy Lie algebras and of homotopy associative algebras on tensor products of algebras of different types and on mapping spaces between coalgebras and algebras. When considering tensor products, such algebraic structures extend the Lie algebra or associative algebra structures that can be obtained by means of the Manin products of operads. These new homotopy algebra structures are proven to be compatible with the concepts of homotopy theory: $\infty$-morphisms and the Homotopy Transfer Theorem. We give a conceptual interpretation of their Maurer{\textendash}Cartan elements. In the end, this allows us to construct the deformation complex for morphisms of algebras over an operad and to represent the deformation $\infty$-groupoid for differential graded Lie algebras.
\end{abstract}

\maketitle

\setcounter{tocdepth}{1}
\tableofcontents

\section{Introduction}

In deformation theory, one essentially studies the spaces of Maurer--Cartan elements of differential graded Lie algebras and homotopy Lie algebras --- which for example play a crucial role in Kontsevich's celebrated proof of deformation quantization of Poisson manifolds \cite{kontsevich03}. We will give a solution to the first of the following two problems and outline a solution for the second one, which is solved in the article \cite{rn17cosimplicial} using the results of this paper in a crucial fashion.
\begin{enumerate}
	\item Given a type of algebras (encoded by an operad), what is the correct homotopy Lie algebra encoding the deformation theory of morphisms between two algebras of this type? This problem was mentioned by Kontsevich in his recent talk at the S\'eminaire Bourbaki \cite{kontsevich17}.
	\item A good model for the space of Maurer--Cartan elements of a homotopy Lie algebra is given by the Deligne--Hinich--Getzler $\infty$-groupoid of \cite{hinich97} and \cite{getzler09}. However, it is a really big object. A smaller, homotopy equivalent Kan complex was introduced by Getzler in \cite{getzler09}, but it is unfortunately difficult to manipulate directly. Is there a reasonably small, homotopically equivalent Kan complex which can be described explicitly?
\end{enumerate}
Trying to solve the first problem, we are rapidly led to mapping spaces between certain coalgebras and algebras, while our approach to solve the second problem requires a homotopy Lie algebra structure on the tensor products of certain algebras.

\medskip

Given two algebras of different types, \emph{a priori} one cannot say much about the algebraic structure induced on their tensor product. For example, there is no canonical ``easy'' structure on the tensor product of two Lie algebras. However, when the two types of algebras are related in a certain way, one is often able to get some kind of structure. One example of interest is when one of the algebras is over a binary quadratic operad, and the other one is over the Koszul dual operad. In this case, one can endow the tensor product with a natural structure of Lie algebra. This structure can in fact be interpreted in terms of the black and white Manin products for operads. A similar story is true for mapping spaces between coalgebras and algebras. In this paper we will generalize these constructions to algebras up to homotopy.

\medskip

We denote by $\L_\infty$ the operad coding homotopy Lie algebras and by $\A_\infty$ the non-symmetric operad coding homotopy (non-symmetric) associative algebras. Let $\Psi:\Q\to\P$ be a morphism of dg operads. Our first important result states that we can naturally associate to $\Psi$ a morphism from $\L_\infty$ to the tensor product of $\P$ with a quasi-free operad related to $\Q$.

\begin{mainthm}
	Let $\Psi:\Q\to\P$ be a morphism of dg operads such that $\Q$ is augmented. There is a morphism of operads
	\[
	\genmaninDual_\Psi:\L_\infty\longrightarrow\hom(\bar(\susp\otimes\Q),\P)
	\]
	which is compatible with compositions in the sense that
	\[
	\genmaninDual_\Psi(\ell) = \Psi\genmaninDual_\Theta(\ell) = \genmaninDual_\Psi(\ell)\bar(\susp\otimes\Theta)
	\]
	for any $\ell\in\L_\infty$.
	
	Dually, if we also suppose that $\Q(n)$ is finite dimensional for all $n\ge0$, there is a morphism of operads
	\[
	\genmanin_\Psi:\L_\infty\longrightarrow\P\otimes\Omega((\susp^{-1})^c\otimes\Q^\vee)\ ,
	\]
	with a similar compatibility with respect to compositions.
\end{mainthm}

When $\Q$ is binary Koszul, the last operad $\Omega((\susp^{-1})^c\otimes\Q^\vee)$ is nothing but the Koszul resolution $\Q^{\shriek}_\infty$ of the Koszul dual of $\Q$. If the operads are non-symmetric, then the morphisms have the operad $\A_\infty$ as domain, instead of the operad $\L_\infty$. In particular, this tells us that if we are given an algebra over the first operad $\P$ and a second algebra over the second operad $\Q^{\shriek}_\infty$, then we can endow their tensor product with a natural structure of an $\L_\infty$-algebra --- or an $\A_\infty$-algebra in the non-symmetric case. We denote by $\otimes^\Psi$ this structure. Dually, we associate to $\Psi$ a morphism from $\L_\infty$ to the convolution operad of a cooperad related to $\Q$ and $\P$. Therefore, the mapping space between a $\bar(\susp\otimes\Q)$-coalgebra and a $\P$-algebra is an $\L_\infty$-algebra, respectively an $\A_\infty$-algebra in the ns case, which we denote by $\hom^\Psi$.

\medskip

An interesting example is the following one. Let $A$ be a Lie algebra, and let $C$ be a commutative algebra up to homotopy. Then we obtain a natural $\L_\infty$-algebra structure on their tensor product $A\otimes C$. This structure has already appeared in the literature in the article \cite{turchin15} by V. Turchin and T. Willwacher on Hochschild--Pirashvili homology.

\medskip

Let us now focus on the case of binary quadratic operads. When working with algebras up to homotopy, there is a natural extension of the notion of morphisms of algebras called $\infty$-morphisms. They are a sensible homotopical generalization of strict morphisms and they play a crucial role in Kontsevich's work \cite{kontsevich03}. We show that our operator $\otimes^\Psi$ is functorial not only with respect to strict morphisms of algebras, but also with respect to $\infty$-morphisms on the second algebra. All together, the above mentioned results give the following theorem.

\begin{thmbifunctor}
	The above mentioned $\L_\infty$-algebra structure on the tensor product of a $\P$-algebra and a $\Q^{\shriek}_\infty$-algebra induces a bifunctor
	\[
	\otimes^\Psi:\P\text{-}\mathsf{alg}\times\infty\text{-}\Q^{\shriek}_\infty\text{-}\mathsf{alg}\longrightarrow\infty\text{-}\L_\infty\text{-}\mathsf{alg}\ ,
	\]
	where $\infty\text{-}\L_\infty\text{-}\mathsf{alg}$ denotes the category of $\L_\infty$-algebras with their $\infty$-morphisms, and similarly for $\infty\text{-}\Q^{\shriek}_\infty\text{-}\mathsf{alg}$.
\end{thmbifunctor}

Next, we study the compatibility of these new structures with an important tool in homotopical algebra: the Homotopy Transfer Theorem, which tells us that, given a retraction of chain complexes
\begin{center}
	\begin{tikzpicture}
		\node (a) at (0,0){$B$};
		\node (b) at (2,0){$C$};
		
		\draw[->] (a)++(.3,.1)--node[above]{\mbox{\tiny{$p$}}}+(1.4,0);
		\draw[<-,yshift=-1mm] (a)++(.3,-.1)--node[below]{\mbox{\tiny{$i$}}}+(1.4,0);
		\draw[->] (a) to [out=-150,in=150,looseness=4] node[left]{\mbox{\tiny{$h$}}} (a);
	\end{tikzpicture}
\end{center}
and a structure of an algebra of a certain kind on $B$, then there is a coherent way to induce a homotopically equivalent structure of the same algebra but now up to homotopy on $C$. Our construction $\otimes^\Psi$ of an $\L_\infty$-algebra structure is compatible with the Homotopy Transfer Theorem in the following sense. Let $A$ be a $\P$-algebra and take a $\Q^{\shriek}$-algebra $B$ together with the data of a retraction from $B$ to a subcomplex $C$. Then there are two natural ways to endow $A\otimes C$ with an $\L_\infty$-algebra structure: one can either pull back the natural $\P\otimes\Q^{\shriek}$-algebra structure on the tensor product $A\otimes B$ to a Lie algebra structure and then use the Homotopy Transfer Theorem, or one can first use the Homotopy Transfer Theorem to obtain a $\Q^{\shriek}_\infty$-algebra structure on $C$ and then pull back the resulting algebraic structure on $A\otimes C$ using Theorem \ref{thm:mainThm}.

\begin{thmtransport}
	The two $\L_\infty$-algebra structures thus obtained on the tensor product $A\otimes C$ are equal.
\end{thmtransport}

Let us now go back to the case of algebras over arbitrary operads. After using Theorem \ref{thm:mainThmDual} or its dual version, Theorem \ref{thm:mainThm}, to endow a mapping space $\hom(D,A)$, respectively a tensor product $A\otimes C$ with an $\L_\infty$-algebra structure, it is a natural question to ask what are the Maurer--Cartan elements of the resulting $\L_\infty$-algebra. To any morphism $\Psi$ of dg operads, one can naturally associate a twisting morphism $\psi\in\Tw(\bar\Q,\P)$, and thus a complete cobar construction
\[
\widehat{\Omega}_\psi:\mathsf{dg}\widehat{\P}\text{-}\mathsf{alg}\longrightarrow\mathsf{dg}\bar\Q\text{-}\mathsf{cog}\ .
\]
This is a modification of a classical construction, where we go from $\bar\Q$-co-al\-ge\-bras to \emph{complete} $\P$-algebras instead than from \emph{conilpotent} $\bar\Q$-coalgebras to $\P$-algebras. Using this, we are able to prove the following statements.

\begin{thmMC}
	Let $A$ be a $\P$-algebra, let $C$ be a finite dimensional $\Q^{\shriek}_\infty$-algebra, and let $D$ be a $\bar(\susp\otimes\Q)$-coalgebra. If $A$ is a complete $\P$-algebra, there is a natural bijection
	\[
	\MC(\hom^\Psi(D,A))\cong\hom_{\mathsf{dg}\P\text{-}\mathsf{alg}}(\widehat{\Omega}_\alpha(s^{-1}D),A)\ .
	\]
	Similarly, if $A$ is any $\P$-algebra, but if we assume that $D$ is conilpotent, then there is a natural bijection
	\[
	\MC(\hom^\Psi(D,A))\cong\hom_{\mathsf{dg}\P\text{-}\mathsf{alg}}(\Omega_\alpha(s^{-1}D),A)\ .
	\]
	Dually, assuming that $A$ is complete, there is a natural bijection
	\[
	\MC(A\otimes^\Psi C)\cong\hom_{\mathsf{dg}\P\text{-}\mathsf{alg}}(\widehat{\Omega}_\alpha(s^{-1}C^\vee),A)\ .
	\]
\end{thmMC}

We use the second bijection to construct the deformation complex for morphisms of $\P$-algebras in Section \ref{subsection:defCplx}. The third one plays a crucial role in the article \cite{rn17cosimplicial}, where we use it to represent the Maurer--Cartan space (i.e. the Deligne--Hinich--Getzler $\infty$-groupoid) of dg Lie algebras. We outline this in Section \ref{subsection:cosimplicial}.

\subsection*{Structure of the paper}

In Section \ref{sect:recollection}, we give a not so short recollection on basic facts and constructions about operads which we use throughout the paper. In Section \ref{sect:sect2}, we state and prove the central theorem of the present article, on which rest all other results. In Section \ref{sect:sect3}, we focus on the case where the operads into play are binary quadratic and we study the compatibility of the main theorem with the Manin products and with $\infty$-morphisms of homotopy algebras. We proceed with Section \ref{sect:sect4}, where we give two ways to use the Homotopy Transfer Theorem to produce the structure of a Lie algebra up to homotopy on certain tensor products of algebras, one using Manin products and the other using our main theorem, and prove that they are actually equal. In Section \ref{sect:MCel}, we go back to general operads and study the set of Maurer--Cartan elements of the Lie algebra up to homotopy produced via the main theorem. In Section \ref{sect:defTheory}, we give two applications to deformation theory. To conclude the main body of the article, we give some explicit examples of applications of the main theorem in the dual case in Section \ref{sect:sect5}. Additionally, in Appendix \ref{sect:appendixCompletePalg} we give some basic definitions and results about complete and filtered algebras over an operad, which we need in Section \ref{sect:MCel}.

\subsection*{Notation and conventions}

Throughout this paper, we work over a fixed field $\k$ of characteristic $0$, with the remarkable exception of whenever we give the results for non-symmetric operads, where we admit any field. This is necessary, for example because we will often need to identify invariants and coinvariants for actions of the symmetric groups. We recall how this is done in general: let $G$ be a finite group, and $V$ a representation of $G$. Then the isomorphism
\[
V^G\longrightarrow V_G
\]
is given by sending an invariant $v$ to $\tfrac{1}{|G|}[v]$, where $[v]$ here denotes the class of $v$ in $V_G$. Conversely, the isomorphism
\[
V_G\longrightarrow V^G
\]
is given by sending $[v]$ to $\sum_{g\in G}g\cdot v$, the sum of the elements of the orbit of a representative of $[v]$ under the action of $G$.

\medskip

We always work over chain complexes unless otherwise specified. In particular, Maurer--Cartan elements of Lie algebras and $\L_\infty$-algebras are of degree $-1$. To take care of signs, we adopt the Koszul convention and the Koszul sign rule, see \cite[Sect. 1.5.3]{vallette12}, and make heavy use of them throughout the paper.  All chain complexes are $\mathbb{Z}$-graded.

\medskip

We denote the symmetric group on $n$ elements by $\S_n$.

\medskip

If $V$ is a chain complex, we denote by $V^\vee$ its linear dual chain complex. It is given by $(V^\vee)_n\coloneqq (V_{-n})^\vee$. Its differential is equal to $d_{V^\vee}\coloneqq -d_V^\vee$, so that the natural pairing
\[
\langle\ ,\ \rangle:V^\vee\otimes V\longrightarrow\k
\]
is a morphism of chain complexes, where the base field $\k$ is seen as a chain complex concentrated in degree $0$. More generally, if $V,W$ are two chain complexes, we will denote by $\hom(V,W)$ the internal hom in chain complexes. Its degree $n$ elements are the linear maps of degree $n$ from $V$ to $W$, and the differential is given by
\[
\partial(\phi)\coloneqq d_W\phi - (-1)^{|\phi|}\phi d_V
\]
on homogeneous elements.

\subsection*{Acknowledgments}

I thank the anonymous referee for his or her useful comments and suggestions. I am as always also extremely grateful to my advisor Bruno Vallette for the constant support, the always relevant comments and corrections, and the many useful discussions.

\section{Recollection on operads} \label{sect:recollection}

In this section, we give a recollection on various notions in operad theory. We make explicit certain objects that we will need in what follows, such as the structure of algebra over the convolution operad for the mapping space of a coalgebra and an algebra. We also introduce some new notations that we will use throughout the paper. We try to stay as close as possible to the conventions of the book \cite{vallette12}.

\subsection{Operads}

For more details about $\S$-modules, the definition of operads, and algebras over operads, see the book \cite[Sect. 5.1--3]{vallette12}.

\begin{definition}
	An $\S$-module over the field $\k$ is a collection
	\[
	M = (M(0),M(1),M(2),\ldots)
	\]
	of right $\k[\S_n]$-modules, for all $n\in\mathbb{N}$. If $\mu\in M(n)$, we say that $\mu$ has \emph{arity} $n$. A morphism of $\S$-modules $f:M\to N$ is a collection of $\S_n$-equivariant maps $f:M(n)\to N(n)$, for all $n\in\mathbb{N}$.
\end{definition}

\begin{definition}
	Let $M,N$ be two $\S$-modules. Their \emph{composite} is the $\S$-module $M\circ N$ given by
	\[
	(M\circ N)(n)\coloneqq \bigoplus_{k\ge0}M(k)\otimes_{\S_k}\left(\bigoplus_{n_1+\cdots + n_k = n}N(n_1)\otimes\cdots\otimes N(n_k)\otimes_{\S_{n_1}\times\cdots\times \S_{n_k}}\k[\S_n]\right)\ .
	\]
\end{definition}

In order to work with the elements of this object, we introduce the following notation. An element of $(M\circ N)(n)$ can always be represented by an element $\mu\in M(k)$, plus $k$ elements $\nu_1,\ldots,\nu_k$ with $\nu_i\in N(n_i)$, and a shuffle $\sigma\in\Sh(n_1,\ldots,n_k)$. We denote the resulting element by
\[
\mu\otimes(\nu_1\otimes\cdots\nu_k)^\sigma,
\]
and write $M\otimes_{k,n_1,\ldots,n_k}^\sigma N$ for the collection of all such elements. We also use the notation
\[
M\otimes_{k,n_1,\ldots,n_k}N\coloneqq\bigoplus_{\sigma\in\Sh(n_1,\ldots,n_k)}M\otimes_{k,n_1,\ldots,n_k}^\sigma N\ .
\]
Notice that the module $(M\circ N)(n)$ is given by the coinvariants with respect to the action of $\S_n$ of the direct sum of all such vector spaces with $n_1+\cdots+n_k = n$ equipped with the obvious action of $\S_n$, that is
\[
(M\circ N)(n) = \bigoplus_{k\ge0}\left(\bigoplus_{n_1+\cdots+ n_k = n}M\otimes_{k,n_1,\ldots,n_k}N\right)_{\S_n}.
\]

\begin{definition}
	The \emph{Hadamard tensor product} of two $\S$-modules $M$ and $N$ is the arity-wise tensor product
	\[
	M\underset{H}{\otimes} N(n)\coloneqq M(n)\otimes N(n)\ .
	\]
\end{definition}

In the rest of this paper, we will omit the $H$ and simply write it as $\otimes$  whenever we are talking about $\S$-modules.

\medskip

We will need the following $\S$-modules:
\begin{itemize}
	\item The \emph{unit $\S$-module} is
	\[
	I = (0,\k,0,0,\ldots)\ .
	\]
	It is the unit for composite product of $\S$-modules.
	\item Associated to any chain complex $V$, there is a canonical $\S$-module $\End_V$ given by
	\[
	\End_V(n)\coloneqq \hom(V^{\otimes n},V)\ .
	\]
\end{itemize}

The category of $\S$-modules is made into a monoidal category by taking the composite of $\S$-modules as monoidal product and the unit $\S$-module $I$ as unit.

\begin{definition}
	An \emph{operad} $\P$ is a monoid in this monoidal category. More explicitly, it amounts to the data of an $\S$-module, denoted again by $\P$, together with a \emph{composition map}
	\[
	\gamma_\P:\P\circ\P\longrightarrow\P
	\]
	and a \emph{unit map}
	\[
	\eta_\P:I\to\P
	\]
	satisfying certain compatibility conditions.
\end{definition}

In this definition, we withheld an \emph{algebraic} or a \emph{dg} before the word operad. They come about when we decide to work with $\S$-modules over graded vector spaces or chain complexes respectively.

\medskip

For any chain complex $V$, the $\S$-module $\End_V$, equipped with the usual composition of functions, is a dg operad.

\subsection{Algebras over an operad}

\begin{definition}
	Let $\P$ be a dg operad. A structure of an \emph{algebra over $\P$} (or a \emph{$\P$-algebra}) on a chain complex $A$ is a morphism of dg operads
	\[
	\P\longrightarrow\End_A\ .
	\]
\end{definition}

Recall that the structure of a $\P$-algebra on a chain complex $A$ is equivalent to a linear morphism
\[
\rho_A:\P(A)\longrightarrow A\ ,
\]
called the \emph{composition} of $A$, making certain diagrams commute (see e.g. \cite[pp. 132--133]{vallette12}).

\begin{proposition}
	The \emph{free $\P$-algebra} generated by a chain complex $V$ is the $\P$-algebra
	\[
	\P(V)\coloneqq\bigoplus_{n\ge0}\P(n)\otimes_{\S_n}V^{\otimes n}.
	\]
\end{proposition}

Being free means satisfying the following universal property: any morphism from a free $\P$-algebra to another $\P$-algebra is completely determined by its values on the generators.

\subsection{Cooperads and coalgebras over a cooperad}

For details about cooperads, see \cite[Sect. 5.8]{vallette12}.

\medskip

Dual to the notion of an operad is the notion of a cooperad. Since we work over a field of characteristic $0$, we always identify invariants and coinvariants. We consider the ``completed'' composite of $S$-modules:
\[
(M\hat{\circ} N)(n)\coloneqq \prod_{k\ge0}M(k)\otimes_{\S_n} \left(\prod_{n_1+\cdots + n_k = n}N(n_1)\otimes\cdots\otimes N(n_k)\otimes_{\S_{n_1}\times\cdots\times\S_{n_k}}\k[\S_n]\right)\ .
\]

\begin{definition}
	A cooperad is an $\S$-module $\C$ together with a \emph{decomposition map} $\Delta_\C:\C\to\C\hat{\circ}\C$ and a \emph{counit map} $\epsilon_\C:\C\to I$ satisfying analogous commutative diagrams. A cooperad is said to be \emph{conilpotent} if the decomposition map splits through $\C\circ\C$.
\end{definition}

Dual to the notion of an algebra over an operad, there is the concept of a coalgebra over a cooperad.

\begin{definition}
	A $\C$-coalgebra $C$ is said to be \emph{conilpotent} if its decomposition map
	\[
	\Delta_C:C\longrightarrow\widehat{\C}(C)\coloneqq\C\hat{\circ} C
	\]
	splits through $\C(C)\coloneqq\C\circ C$.
\end{definition}

Under the assumption that the underlying $\S$-module is finite dimensional in every arity, the linear dual of an operad becomes a cooperad, while the linear dual of a cooperad is always an operad. Let $\P$ be such an operad. We will use the notation
\[
\Delta^{k,n_1,\ldots,n_k,\sigma}:\P^\vee\longrightarrow(\P\otimes_{k,n_1,\ldots,n_k}^\sigma\P)^\vee\cong\P^\vee\otimes_{k,n_1,\ldots,n_k}^\sigma\P^\vee
\]
for the dual of the restriction of the composition map of $\P$ to $\P\otimes_{k,n_1,\ldots,n_k}^\sigma\P$. If $n_1+n_2 = n+1$, $1\le j\le n_1$, and $\sigma\in\Sh(n_1-1,n_2)$, then we use the notation $\otimes_{n_1,n_2,j}^\sigma$ to mean $\otimes_{n_1,1,\ldots,1,n_2,1,\ldots,1}^{\tilde{\sigma}}$, where the $n_2$ is placed after $j-1$ ones, and where $\tilde{\sigma}$ acts with the second of the two blocks of $\sigma$ on the $n_2$ slots on the second level of the two-level tree, and with the other block on the other variables. Similarly, $\Delta^{n_1,n_2,j,\sigma}$ denotes the respective restriction of $\Delta$.

\subsection{The convolution operad and algebras over it}

If $\C$ is a cooperad and $\P$ is an operad, then the $\S$-action given by conjugation makes
\[
\hom(\C,\P)(n)\coloneqq\hom(\C(n),\P(n))
\]
into an $\S$-module. It is an operad --- the \emph{convolution operad} --- in a natural way as detailed in \cite[Sect. 6.4.1]{vallette12}. The composition map is defined as follows. If we have $\mu\in\hom(\C,\P)(k)$ and $\nu_i\in\hom(\C,\P)(n_i)$ for $1\le i\le k$, then $\gamma_{\hom(\C,\P)}(\mu\otimes(\nu_1\otimes\cdots\otimes\nu_k))$ is the element of $\hom(\C,\P)(n_1+\cdots+n_k)$ obtained by first applying $\Delta_\C$ to $\C(n_1+\cdots+n_k)$, then projecting on the subspace of $\C\circ\C$ with the underlying tree given by the $k$-corolla on the first level and the $n_1$-corolla, the $n_2$-corolla and so on on the second level, applying $\mu$ at the first level and $(\nu_1,\ldots,\nu_k)$ at the second level, and finally composing with $\gamma_\P$. There is a passage from invariant to coinvariants here: the decomposition $\Delta_\C$ lands in invariants, but the composition $\gamma_\P$ takes coinvariants as argument. This is \emph{not} done using the isomorphism described in the introduction, but by identifying the invariants with a subspace of the tensor product and simply taking the equivalence class.

\medskip

In Section \ref{sect:MCel}, we will be interested in certain algebras over this operad. Namely, let $D$ be a $\C$-coalgebra, and let $A$ be a $\P$-algebra. Let $\mu\in\hom(\C,\P)(n)$ and $\varphi_1,\ldots,\varphi_n\in\hom(D,A)$. We define
\[
\gamma_{\hom(D,A)}(\mu\otimes(\varphi_1\otimes\cdots\otimes\varphi_n))
\]
to be the composite
\[
D\xrightarrow{\Delta_D(n)}(\C(n)\otimes D^{\otimes n})^{\S_n}\xrightarrow{\mu\otimes(\varphi_1\otimes\cdots\otimes\varphi_n)}(\P(n)\otimes A^{\otimes n})^{\S_n}\longrightarrow\P(n)\otimes_{\S_n}A^{\otimes n}\xrightarrow{\gamma_A}A\ ,
\]
where once again the passage from invariants to coinvariants in \emph{not} done using the isomorphism of the introduction, but as described above. Here, $\Delta_D(n)$ is the composite of $\Delta_D$ followed by the projection onto $\C(n)\otimes D^{\otimes n}$.

\begin{proposition} \label{prop:convolutionAlg}
	 The map $\gamma_{\hom(D,A)}$ makes the chain complex $\hom(D,A)$ into a $\hom(\C,\P)$-algebra. 
\end{proposition}

\begin{proof}
	It is straightforward to check the axioms for an algebra over an operad. We leave the explicit computations to the reader.
\end{proof}

\subsection{Free operads} \label{subsect:freeOperads}

\begin{definition}
	Let $M$ be an $\S$-module. A \emph{free operad} over $M$ is an operad $\T(M)$ together with a morphism of $\S$-modules $\eta(M):M\to\T(M)$ satisfying the following universal property: for any operad $\P$, every morphism $f:M\to\P$ of $\S$-modules extends uniquely to a morphism of operads $\tilde{f}:\T(M)\to\P$ such that $\tilde{f}\eta(M) = f$.
\end{definition}

As usual with this sort of universal properties, the free operad over an $\S$-module is unique up to isomorphism. It can be explicitly described as follows.

\begin{definition}
	Let $M$ be an $\S$-module. The \emph{tree module} $\T(M)$ over $M$ is the $\S$-module spanned in arity $n$ by rooted trees of arity $n$ with vertices of arity $k$ labeled by elements of $M(k)$, together with the obvious right $\S_n$-action.
\end{definition}

\begin{theorem}
	Let $M$ be an $\S$-module. The free operad over $M$ is the $\S$-module $\T(M)$ together with the grafting of trees as composition map.
\end{theorem}

In the dual picture, the \emph{cofree cooperad} over $M$ is the cooperad $\T^c(M)$ given again by the tree module module $\T(M)$ endowed with the decomposition of trees as decomposition map. It satisfies the dual universal property to the one for free operads.

\medskip

If $\P$ is an operad, by iterated application if the composition map we obtain a morphism of operads
\[
\T(\P)\longrightarrow\P\ ,
\]
which we denote by $\tilde{\gamma}_\P$. This yields the \emph{monadic} definition of an operad. Dually, if $\C$ is a cooperad, repeated applications of the decomposition map yield a map
\[
\C\longrightarrow\T^c(\C)\ ,
\]
which we denote by $\tilde{\Delta}_\C$.

\medskip

Given two $\S$-modules $M$ and $N$, there is a natural morphism of operads
\[
\Phi:\T(M\otimes N)\longrightarrow\T(M)\otimes\T(N)\ .
\]
It is induced by the map sending a tree with vertices indexed by elements of $M\otimes N$ to the tensor product of two copies of the same tree, the first one with the vertices labeled by the respective elements of $M$, and the second one with the vertices labeled with the elements of $N$, see \cite[pp. 308--309]{vallette12}. For $\tau$ a rooted tree, denote by $\T(M)^\tau$ the subspace of $\T(M)$ spanned by elements having $\tau$ as the underlying tree. Then we can define a map
\[
T^\tau:\T(M)^\tau\otimes\T(N)^\tau\longrightarrow\T(M\otimes N)^\tau
\]
which is inverse to the restriction of $\Phi$ to $\T(M\otimes N)^\tau$. We often refer to it as the \emph{switch map}.

\subsection{Operadic suspensions}

We use the letter $s$ to denote a formal element of degree $1$, and denote by $s^{-1}$ its dual. Therefore, if $V$ is a graded vector space, its suspension is given by $sV$, i.e. $(sV)_n \cong V_{n-1}$. Notice that the dual of $s\otimes s$ is then $-s^{-1}\otimes s^{-1}$, and that $s^{-1}s \cong 1 \cong -ss^{-1}$. We denote by $s^{-2}$ the element $s^{-1}\otimes s^{-1}$.

\medskip

Let $\P$ be an operad. Then the suspension $s\P$ is not an operad in general. However, there is an operadic version of suspension. Let $\susp\coloneqq\End_{\k s}$ be the operad which is $1$-dimensional in every arity, where it is spanned by the degree $1-n$ map $\susp_n$ sending $s^{\otimes n}$ to $s$. Similarly, we define $\susp^{-1}\coloneqq\End_{\k s^{-1}}$, and we denote by $\susp^c$ the dual cooperad of $\susp^{-1}$, and by $(\susp^{-1})^c$ the dual cooperad of $\susp$.

\begin{definition}
	Let $\P$ be an operad. The \emph{operadic suspension} of $\P$ is the operad $\susp\otimes\P$. Similarly, the \emph{operadic desuspension} of $\P$ is $\susp^{-1}\otimes\P$.
	
	Let $\C$ be a cooperad. Analogously to the above, the \emph{operadic suspension} of $\C$ is the cooperad $\susp^c\otimes\C$, and its \emph{operadic desuspension} is $(\susp^{-1})^c\otimes\C$.
\end{definition}

Notice that on the level of the underlying $\S$-modules, we have $\susp^c\cong\susp$ and similarly $\susp^{-1}\cong(\susp^{-1})^c$. However, there is a sign $(-1)^\epsilon$ due to the Koszul sign convention appearing in the isomorphism. The correct sign can be found by the following simple computation:
\begin{align*}
	(-1)^{\frac{n(n-1)}{2}} =&\ s^{-n}s^n\\
	=&\ (-1)^\epsilon\susp_n^{-1}\susp_ns^{-n}s^n\\
	=&\ (-1)^\epsilon\susp_n^{-1}s^{-n}\susp_ns^n\\
	=&\ (-1)^\epsilon s^{-1}s\\
	=&\ (-1)^\epsilon,
\end{align*}
where in the third line we used the fact that $(n-1)n$ is always even. Here --- and in the rest of this paper as well --- certain tensor products, maps and identifications are left implicit. For example, in various places we identified $1\in\k$ with the element of $V\otimes V^\vee$ corresponding to the identity of $V$, where $V$ is a finite dimensional vector space.

\subsection{Bar-cobar adjunction for operads} \label{subsect:barCobarOperads}

There is a pair of adjoint functors between conilpotent coaugmented dg cooperads and augmented dg operads, called the bar and the cobar construction, respectively. See \cite[Sect. 6.5.1--3]{vallette12} for details.

\begin{definition}
	The \emph{cobar construction} is the functor $\Omega$ taking a conilpotent coaugmented dg cooperad $\C$ and giving the quasi-free dg operad
	\[
	\Omega\C\coloneqq(\T(s^{-1}\overline{\C}),d\coloneqq d_1+d_2)\ ,
	\]
	where $d_1$ is the unique derivation extending the differential of $\C$, and $d_2$ is the unique derivation extending (a suspended version of) the infinitesimal decomposition map $\Delta_{(1)}$.
\end{definition}

\begin{definition}
	Dually, the \emph{bar construction} is the functor $\bar$ taking an augmented dg operad $\P$ and giving back the quasi-free dg cooperad
	\[
	\bar\P\coloneqq(\T^c(s\overline{\P}),d\coloneqq d_1+d_2)\ ,
	\]
	where $d_1$ is the unique coderivation extending the differential of $\P$, and $d_2$ is the unique coderivation extending (a suspended version of) the composition map of $\P$ on trees with two vertices.
\end{definition}

Let $\C$ be a dg cooperad and let $\P$ be a dg operad, then a \emph{twisting morphism} from $\C$ to $\P$ is a degree $-1$ morphism of $\S$-modules from $\C$ to $\P$ satisfying a certain version of the Maurer--Cartan equation. The set of all twisting morphisms is denoted by $\Tw(\C,\P)$, see \cite[Sect. 6.4]{vallette12}.

\begin{theorem}\label{thm:rosettaStone}
	Let $\C$ be a conilpotent, coaugmented cooperad, and let $\P$ be an augmented operad. There are natural isomorphisms
	\[
	\hom_{\mathsf{dg\ Op}}(\Omega\C,\P)\cong\Tw(\C,\P)\cong\hom_{\mathsf{dg\ coOp}}(\C,\bar\P)\ .
	\]
	In particular, the functors $\Omega$ and $\bar$ form an adjoint pair.
\end{theorem}

\subsection{Binary quadratic operads, Koszul duality and homotopy algebras}

\begin{definition}
	An \emph{operadic quadratic data} is a couple $(E,R)$, where $E$ is an $\S$-module and $R$ is a sub-$\S$-module of $\T(E)^{(2)}$, the $\S$-module of trees with two vertices. If $(E,R)$ and $(F,S)$ are two operadic quadratic data, a morphism
	\[
	f:(E,R)\longrightarrow(F,S)
	\]
	of operadic quadratic data between them is a morphism of $\S$-modules $f:E\to F$ such that $\T(f)(R)\subseteq S$.
\end{definition}

To such a quadratic data, one can associate a \emph{quadratic operad} $\P(E,R)$ by
\[
\P(E,R)\coloneqq\T(E)/(R)\ .
\]
Such an operad is said to be \emph{binary} if the generating $\S$-module $E$ is concentrated in arity $2$. The category of binary quadratic operads is the subcategory of the category of operads having as objects the quadratic operads and as morphisms the morphisms of operads induced by morphisms of operadic quadratic data. Dually, we can also associate a \emph{quadratic cooperad} $\C(E,R)$ to any operadic quadratic data. It is defined through a universal property and does not admit a full description as simple as the one for $\P(E,R)$, see \cite[Sect. 7.1]{vallette12}.

\medskip

In the category of quadratic operads, we have the very useful tool of Koszul duality.

\begin{definition}
	Given a quadratic operad $\P\coloneqq\P(E,R)$, then its \emph{Koszul dual cooperad} is the quadratic cooperad
	\[
	\P^{\antishriek}\coloneqq\C(sE,s^2R)\ .
	\]
\end{definition}

\begin{definition}
	There is also the notion of the \emph{Koszul dual operad} of $\P$, which is defined as
	\[
	\P^{\shriek}\coloneqq\big(\susp^c\otimes\P^{\antishriek}\big)^\vee.
	\]
\end{definition}

The following computation will be useful later on. We have
\[
\P^{\antishriek} \cong (\susp^{-1})^c\otimes\big(\P^{\shriek}\big)^\vee,
\]
therefore, if we take $\P = \Q^{\shriek}$ for some quadratic operad $\Q$, then
\begin{equation}\label{eq:shriek-antishriek}
	\big(\Q^{\shriek}\big)^{\antishriek} = (\susp^{-1})^c\otimes\Q^\vee.
\end{equation}
Another useful fact is the following one.

\begin{lemma}\label{lemma:reprPshriek}
	Let $\P\coloneqq\P(E,R)$ be a binary quadratic operad. Then $\P^{\shriek}$ is again binary quadratic, and has the explicit presentation
	\[
	\P^{\shriek} = \P(s^{-1}\susp_2^{-1}E^\vee,R^\perp)\ .
	\]
\end{lemma}

The category of dg operads admits a model structure, see for example \cite{hinich97homological}, and it can be shown that, for a dg operad $\P$ satisfying certain assumptions, the dg operad $\Omega\bar\P$ is a cofibrant resolution of $\P$. However, the dg operad $\Omega\bar\P$ is a very big object. Thus we are often interested to find smaller cofibrant resolutions. The most common one is the \emph{minimal model} for $\P$ (see \cite[Sect. 6.3.4]{vallette12}). A quadratic operad is said to be \emph{Koszul} if it satisfies certain homological conditions, see \cite[Sect. 7.4 and 8]{vallette12}.

\begin{theorem}[\cite{ginzburg94} and \cite{getzler94}]
	When $\P$ is a Koszul operad, the dg operad $\P_\infty\coloneqq\Omega\P^{\antishriek}$ is the minimal model for $\P$.
\end{theorem}

\begin{definition}
	Let $\P$ be a Koszul operad. A \emph{homotopy $\P$-algebra} is an algebra over the operad $\P_\infty$.
\end{definition}

\begin{remark}
	Many quadratic operads of interest are Koszul. Some examples are the so-called three graces $\lie$, $\com$, and $\ass$, which will be introduced in Section \ref{subsect:mainOperads}.

	Abusing notation, we denote by $\P_\infty$ the dg operad $\Omega\P^{\antishriek}$ even when $\P$ is not Koszul. The operad $\P_\infty$ is cofibrant, but the homotopy theory of $\P_\infty$-algebras is not the same as the homotopy theory of $\P$-algebras when $\P$ is not Koszul.
\end{remark}

A direct consequence of Theorem \ref{thm:rosettaStone} is that the structure of a $\P_\infty$-algebra on a chain complex $A$ is equivalent to the data of a twisting morphism
\[
\varphi_A\in\Tw(\P^{\antishriek},\End_A)\ .
\]

\subsection{Infinity-morphisms}

Let $\P$ be a quadratic operad, and let $A$ be a $\P_\infty$-algebra. One can endow the cofree coalgebra $\P^{\antishriek}(A)$ with the unique coderivation $d_{\P^{\antishriek}(A)}$ extending $d_1 + d_2$, where
\begin{align*}
	d_1 \coloneqq&\ \big(\P^{\antishriek}(A)\xrightarrow{\text{proj}}A\xrightarrow{d_A}A\big)\ ,\\
	d_2 \coloneqq&\ \big(\P^{\antishriek}(A)\xrightarrow{\Delta_{(1)}}(\P^{\antishriek}\circ_{(1)}\P^{\antishriek})(A)\xrightarrow{1\circ_{(1)}\gamma_A}\P^{\antishriek}(A)\xrightarrow{\text{proj}}A\big)\ .
\end{align*}
This is the (canonical) bar construction $\bar_\iota A$ for the $\P_\infty$-algebra $A$.

\begin{remark}
	This construction is the \emph{relative bar construction} of a $\P_\infty$-algebra with respect to the canonical twisting morphism
	\[
	\iota:\P^{\antishriek}\longrightarrow\Omega\P^{\antishriek}=\P_\infty\ ,
	\]
	see \cite[Sect. 6.5.4, 11.2]{vallette12} for more details.
\end{remark}

\begin{definition}
	An $\infty$-morphism of $\P_\infty$-algebras from $A$ to $A'$ is a morphism of differential graded $\P^{\antishriek}$-coalgebras
	\[
	\bar_\iota A\longrightarrow\bar_\iota A'\ .
	\]
	The composition of $\infty$-morphisms is the composition of morphisms of $\P^{\antishriek}$-coalgebras. We use the notation $A\rightsquigarrow A'$ to represent an $\infty$-morphism from $A$ to $A'$. The category of $\P_\infty$-algebras with $\infty$-morphisms is denoted by $\infty\text{-}\P_\infty\text{-}\mathsf{alg}$.
\end{definition}

The data of an $\infty$-morphism $g:A\rightsquigarrow A'$ is equivalent to a collection of maps
\[
g_n:\P^{\antishriek}(n)\otimes_{\S_n}A^{\otimes n}\longrightarrow A'
\]
for all $n\ge0$ satisfying certain relations.

\medskip

While this might seem a peculiar notion at first, it is an important generalization of strict morphisms of algebras. For example, $\infty$-morphisms of $\L_\infty$-algebras play a fundamental role in Kontsevich's proof of deformation quantization of Poisson manifolds in \cite{kontsevich03} and are generally an object of interest in deformation theory.

\subsection{The Homotopy Transfer Theorem}

Let $\P$ be a Koszul operad, and let $X$ be a $\P_\infty$-algebra. If there is a retraction
\begin{center}
	\begin{tikzpicture}
		\node (a) at (0,0){$X$};
		\node (b) at (2,0){$Y$};
		
		\draw[->] (a)++(.3,.1)--node[above]{\mbox{\tiny{$p$}}}+(1.4,0);
		\draw[<-,yshift=-1mm] (a)++(.3,-.1)--node[below]{\mbox{\tiny{$i$}}}+(1.4,0);
		\draw[->] (a) to [out=-150,in=150,looseness=4] node[left]{\mbox{\tiny{$h$}}} (a);
	\end{tikzpicture}
\end{center}
of chain complexes, then the Homotopy Transfer Theorem tells us the following.

\begin{theorem}[Homotopy Transfer Theorem]
	The chain complex $Y$ inherits a $\P_\infty$-algebra structure from $X$ such that $X$ and $Y$ are homotopy equivalent as $\P_\infty$-algebras. Explicitly, the maps $i$ and $p$ can be extended to $\infty$-quasi-isomorphisms $i_\infty$ and $p_\infty$ of $\P_\infty$-algebras between them.
\end{theorem}

This deep theorem can be found in the book \cite[Sect. 10.3]{vallette12}, where its long history and many facets are also elucidated. An explicit formula for the transferred structure is also given as follows. Let
\[
\varphi_X\in\Tw(\P^{\antishriek},\End_X)
\]
be a $\P_\infty$-algebra structure on $X$. Then a $\P_\infty$-algebra structure on $Y$ is given by the twisting morphism $\varphi_Y\in\Tw(\P^{\antishriek},\End_Y)$ defined as the composite
\[
\P^{\antishriek}\xrightarrow{\Delta_{\P^{\antishriek}}}\T^c\left(\overline{\P}^{\antishriek}\right)\xrightarrow{\T^c(s\varphi_X)}\T^c(s\End_X)\xrightarrow{\vdw_Y}\End_Y\ ,
\]
where the map $\vdw_Y$, called the \emph{Van der Laan map}, is given by sending a tree with the vertices marked by elements of $s\End_X$ to the same tree with the suspension $s$ removed, and with $i$'s on the leaves, $p$ at the root, and $h$ applied to every inner edge, and then taking the obvious composition in $\End_Y$. This explicit expression for the transferred structure first appeared in \cite{GCTV12}. A detailed exposition of this map is given in \cite[Sect. 10.3.2]{vallette12}. Explicit formul{\ae} for $i_\infty$ and $p_\infty$ are given in \cite[Sect. 10.3.5--6]{vallette12}.

\subsection{Operads of main interest} \label{subsect:mainOperads}

The main symmetric dg operads appearing in this article are $\lie$, $\com$, and $\ass$, coding Lie, commutative, and associative algebras respectively. They will be described in more detail in Section \ref{subsect:notationsExamples}. Their Koszul resolutions are the dg operads $\L_\infty$, $\C_\infty$, and $\ass_\infty$ respectively. In the ns setting, we consider principally $\as$, also coding associative algebras. Its Koszul resolution is denoted by $\A_\infty$.

\medskip

As already mentioned in the introduction, the operad of main interest to us is the quasi-free dg operad $\L_\infty$ encoding homotopy Lie algebras (also known as strong homotopy Lie algebras in the literature). It is given by
\[
\L_\infty\coloneqq \Omega\lie^{\antishriek}\ .
\]
It is a well known fact that $\lie^{\antishriek} \cong (\susp^{-1})^c\otimes\com^\vee$. For each $n\ge2$, we denote by $\mu_n$ the generating operation of $\com(n)$, corresponding to the only way to multiply $n$ elements in a commutative algebra. We use the notation
\[
\ell_n\coloneqq s^{-1}\susp_n^{-1}\mu_n^\vee\ ,\qquad n\ge2\ ,
\]
for the generators of the operad $\L_\infty$.

\medskip

Similarly, in the non-symmetric (ns) case we will be interested in the quasi-free dg ns operad $\A_\infty$ coding homotopy associative algebras. It is given by
\[
\A_\infty\coloneqq\Omega\as^{\antishriek}\ .
\]
Analogously to what happens in the symmetric case, we have $\as^{\antishriek} \cong (\susp^{-1})^c\otimes\as^\vee$. For each $n\ge2$, we denote by $m_n$ the generating operation in $\as(n)$, which corresponds to the unique way to multiply $n$ elements in an associative algebra (without changing their order). We use the notation
\[
a_n\coloneqq s^{-1}\susp^{-1}_nm_n^\vee\ ,\qquad n\ge2\ ,
\]
for the generators of $\A_\infty$.

\section{Natural homotopy algebra structures on mapping spaces and tensor products} \label{sect:sect2}

In this section, we state and prove the theorem which is the starting point for all other results in this article. It says that if we are given an algebra over an operad and a coalgebra over a cooperad, and if the operad and the cooperad are related in a certain way by a morphism of dg operads, then we can put a natural homotopy Lie algebra on the space of linear morphisms from the coalgebra to the algebra. Dually, if we are given two algebras over dg operads that are related by a morphism of dg operads, then we can put a natural $\L_\infty$-algebra structure on their tensor product.

\subsection{Natural $\L_\infty$-algebra structures}

Let $\P$ and $\Q$ be dg operads, let $\Q$ be augmented, and suppose we have a morphism of dg operads
\[
\Psi:\Q\longrightarrow\P\ .
\]
Denote by $\Psi(n):\Q(n)\to\P(n)$ the restriction of $\Psi$ to arity $n$. Then we can associate to $\Psi$ a map
\[
s^{-1}\lie^{\antishriek}\longrightarrow\hom(\bar(\susp\otimes\Q),\P)
\]
by sending $\ell_n$ to the element
\[
s^{-1}\susp_n^{-1}\Psi(n)\in\hom_\k(s\susp_n\Q(n),\P(n))
\]
given by
\[
(s^{-1}\susp_n^{-1}\Psi(n))(s\susp_nq) = (-1)^{n-1 + \frac{n(n-1)}{2}}\Psi(q)
\]
and then precomposing with the projection
\[
\mathrm{proj}^{(1)}:\bar(\susp\otimes\Q)(n)\longrightarrow s\susp_n\Q(n)
\]
onto the weight $1$ part to get an element of $\hom(\bar(\susp\otimes\Q),\P)(n)$. We denote by
\[
\genmaninDual_\Psi:\L_\infty\longrightarrow\hom(\bar(\susp\otimes\Q),\P)
\]
the unique morphism of algebraic operads extending the map given above. In can be described explicitly as follows. Let $\ell$ be an element of $\L_\infty$ with underlying rooted tree $\tau$, then $\genmaninDual_\Psi$ sends $\ell$ to the morphism from $\bar(\susp\otimes\Q)$ to $\P$ given by first projecting to the submodule $\T(s\susp\otimes\overline{\Q})^\tau$ spanned by elements having $\tau$ as underlying tree, then applying $\genmaninDual_\Psi(\ell_n)$ to the vertices of arity $n$ (with the correct signs appearing because of the Koszul sign rule), and finally composing the resulting tree in $\P$.

\medskip

Without further ado, we can now state the main theorem of this section.

\begin{theorem} \label{thm:mainThmDual}
	Let $\Psi:\Q\to\P$ be a morphism of dg operads. Then the morphism of algebraic operads
	\[
	\genmaninDual_\Psi:\L_\infty\longrightarrow\hom(\bar(\susp\otimes\Q),\P)
	\]
	described above commutes with the differentials, i.e. it is a morphism of dg operads. Moreover, it is compatible with compositions in the following sense. If we have a second morphism $\Theta:\R\to\Q$ of dg operads, then we have the following commutative diagram
	\begin{center}
		\begin{tikzpicture}
			\node (a) at (0,0){$\L_\infty$};
			\node (b) at (4,2){$\hom(\bar(\susp\otimes\R),\Q)$};
			\node (c) at (4,0){$\hom(\bar(\susp\otimes\R),\P)$};
			\node (d) at (4,-2){$\hom(\bar(\susp\otimes\Q),\P)$};
			
			\draw[->] (a)--node[above]{$\genmaninDual_\Theta$}(b);
			\draw[->] (a)--node[above]{$\genmaninDual_{\Psi\Theta}$}(c);
			\draw[->] (a)--node[below]{$\genmaninDual_\Psi$}(d);
			
			\draw[->] (b)--node[right]{$\Psi_*$}(c);
			\draw[->] (d)--node[right]{$\bar(\susp\otimes\Theta)^*$}(c);
		\end{tikzpicture}
	\end{center}
	in the category of dg operads, that is:
	\[
	\genmaninDual_\Psi(\ell) = \Psi\genmaninDual_\Theta(\ell) = \genmaninDual_\Psi(\ell)\bar(\susp\otimes\Theta)
	\]
	for any $\ell\in\L_\infty$.
\end{theorem}

\begin{remark}
	The idea of this construction was already present in Ginzburg--Kapranov \cite[Prop. 3.2.18]{ginzburg94} and elaborated a bit in \cite[Appendix C]{brown15}.
\end{remark}

\begin{remark}
	Thanks to the compatibility with the compositions, it is often only necessary to compute $\genmaninDual_\Q\coloneqq\genmaninDual_{\id_\Q}$ or $\genmaninDual_\P$ in order to find $\genmaninDual_\Psi$ for $\Psi:\Q\to\P$. Indeed, one can write $\Psi = \id_\P\Psi$ or $\Psi = \Psi\id_\Q$ and then use the relations given above.
\end{remark}

\begin{remark}
	A slightly more general construction can be made in an analogous way, as remarked independently in \cite{wierstra16}. Where our results construct a morphism from the operad $\L_\infty$ to a certain convolution operad starting from a morphism of dg operads, in loc. cit. the same result is obtained starting from a twisting morphism. One passes from the former to the latter by pulling back by the canonical twisting morphism $\pi:\bar\Q\to\Q$. More precisely, let $\C$ be a dg cooperad and let $\P$ be a dg operad. Then by definition
	\[
	\Tw(\C,\P) \coloneqq \MC\left(\hom(\C,\P)\right)\ ,
	\]
	where $\hom(\C,\P)$ here denotes the pre-Lie algebra associated to the convolution algebra. As dg operads, we have
	\begin{align*}
		\hom(\C,\P)&\ \cong\hom(\com^\vee,\hom(\C,\P))\\
		&\ \cong\hom((\susp^{-1})^c\otimes\com^\vee,(\susp^{-1})^c\otimes\hom(\C,\P))\\
		&\ \cong\hom((\susp^{-1})^c\otimes\com^\vee,\hom(\susp^c\otimes\C,\P))\ ,
	\end{align*}
	where all the isomorphisms are canonical. Therefore,
	\begin{align*}
		\Tw(\C,\P)\cong&\ \MC\left(\hom((\susp^{-1})^c\otimes\com^\vee,\hom(\susp^c\otimes\C,\P))\right)\\
		=&\ \Tw\left((\susp^{-1})^c\otimes\com^\vee,\hom(\susp^c\otimes\C,\P)\right)\\
		\cong&\ \hom_{\mathsf{dg\ Op}}(\L_\infty,\hom(\susp^c\otimes\C,\P))\ ,
	\end{align*}
	where in the last line we used Theorem \ref{thm:rosettaStone}. In our situation, we can then consider $\C=\bar\Q$ and the twisting morphism
	\[
	\psi = \big(\bar\Q\stackrel{\pi}{\longrightarrow}\Q\stackrel{\Psi}{\longrightarrow}\P\big)\ ,
	\]
	which gives back Theorem \ref{thm:mainThmDual} thanks to the canonical isomorphism $\susp^c\otimes\bar\Q\cong\bar(\susp\otimes\Q)$. The explicit formul{\ae} for the morphisms of operads associated to an arbitrary twisting morphism are analogous to the ones presented above.
\end{remark}

\begin{proof}
	To show that $\genmaninDual_\Psi$ commutes with the differentials, it is enough to check on the generators $\ell_n$ of the operad $\L_\infty$. The differential of the range of $\genmaninDual_\Psi$ is given by
	\[
	d_{\hom(\bar(\susp\otimes\Q),\P)} = (d_\P)_* - (d_1+d_2)^*,
	\]
	with $d_1$ and $d_2$ as described in Subsection \ref{subsect:barCobarOperads}. Notice that $d_\P\genmaninDual_\Psi(\ell_n)$ and $\genmaninDual_\Psi(\ell_n)d_1$ can be non-zero only on elements of weight $1$ of $\bar(\susp\otimes\Q)$, while $\genmaninDual_\Psi(\ell_n)d_2$ vanishes everywhere except on the weight $2$ part $\bar(\susp\otimes\Q)\circ_{(1)}\bar(\susp\otimes\Q)$, as does $\genmaninDual_\Psi(d_{\L_\infty}\ell_n)$. Let $s\susp_nq\in\bar(\susp\otimes\Q)^{(1)}$, then we compute
	\begin{align*}
		\big(d_\P\genmaninDual_\Psi(\ell_n) -&\ (-1)^n\genmaninDual_\Psi(\ell_n)d_1\big)(s\susp_nq) =\\
		=&\ d_\P\left((-1)^{n-1+\frac{n(n-1)}{2}}\Psi(q)\right) - (-1)^n\genmaninDual(\ell_n)\big((-1)^ns\susp_nd_\Q q\big)\\
		=&\ (-1)^{n-1+\frac{n(n-1)}{2}}\big(d_\P\Psi(q) - \Psi(d_\Q q)\big)\\
		=&\ 0\ .
	\end{align*}
	Here we used the fact that $\Psi$ is a morphism of dg operads, and thus commutes with the differentials $d_\P$ and $d_\Q$. To complete this part of the proof, we only have to check what happens on elements of weight $2$. So we consider
	\[
	\left(s\susp_{n_1}q_1\otimes_js\susp_{n_2}q_2\right)^\sigma\in\bar(\susp\otimes\Q)\otimes_{n_1,n_2,j}^\sigma\bar(\susp\otimes\Q)\ ,
	\]
	where $n_1+n_2 = n+1$ and $\sigma\in\Sh(n_1-1,n_2)$. Then we have
	\begin{align*}
		-(-1)^n\genmaninDual_\Psi(\ell_n)d_2&\left(s\susp_{n_1}q_1\otimes_js\susp_{n_2}q_2\right)^\sigma =\\
		=&\ (-1)^{n+1}\genmaninDual_\Psi(\ell_n)\left((-1)^{n_2|q_1|+n_1-1+(j-1)(n_2-1)+\sigma}s\susp_n\gamma_\Q(q_1\otimes_jq_2)^\sigma\right)\\
		=&\ (-1)^{n_2|q_1|+n_1-1+(j-1)(n_2-1)+\sigma + \frac{n(n-1)}{2}}\Psi(\gamma_\Q(q_1\otimes_jq_2)^\sigma)\ .
	\end{align*}
	The signs in the second line appear because of switches, the composition in $\susp$, and because $\susp_n$ carries the sign representation of $\S_n$. At the same time,
	\begin{align*}
		\genmaninDual_\Psi(d_{\L_\infty}&\ell_n)\left(s\susp_{n_1}q_1\otimes_js\susp_{n_2}q_2\right)^\sigma =\\
		=&\ \genmaninDual_\Psi\left(\sum_{\substack{\tilde{n}_1+\tilde{n}_2 = n-1\\\tilde{\sigma + 1}\in\Sh(\tilde{n}_1-1,\tilde{n}_2)}}(-1)^{\tilde{\sigma} + \tilde{n}_1}(\ell_{\tilde{n}_1}\otimes_1\ell_{\tilde{n}_2})^{\tilde{\sigma}}\right)\left(s\susp_{n_1}q_1\otimes_js\susp_{n_2}q_2\right)^\sigma\\
		=&\ \genmaninDual_\Psi\left((-1)^{(j-1)(n_2-1)+\sigma + n_1}(\ell_{n_1}\otimes_j\ell_{n_2})^\sigma\right)\left(s\susp_{n_1}q_1\otimes_js\susp_{n_2}q_2\right)^\sigma\\
		=&\ (-1)^{(j-1)(n_2-1)+\sigma + n_1}\gamma_\P\left((-1)^{n_2(n_1+|q_1|)}\genmaninDual_\Psi(\ell_{n_1})(s\susp_{n_1}q_1)\otimes\genmaninDual_\Psi(\ell_{n_2})(s\susp_{n_2}q_2)\right)^\sigma\\
		=&\ (-1)^{(j-1)(n_2-1)+\sigma + n_1 + n_2|q_1| + n_1n_2}\gamma_\P\left((-1)^{n_1-1+\frac{n_1(n_1-1)}{2}}\Psi(q_1)\otimes_j(-1)^{n_2-1+\frac{n_2(n_2-1)}{2}}\Psi(q_2)\right)^\sigma\\
		=&\ (-1)^{(j-1)(n_2-1)+\sigma + n_1 + n_2|q_1| + \frac{n(n-1)}{2}}\gamma_\P\left(\Psi(q_1)\otimes_j\Psi(q_2)\right)^\sigma.
	\end{align*}
	Comparing the signs and using the fact that $\Psi$ is a morphism of operads, we see that the two expressions are equal. Therefore, the morphism of algebraic operads $\genmaninDual_\Psi$ commutes with the differentials. Checking the compatibility with compositions is straightforward and left to the reader.
\end{proof}

\begin{corollary} \label{cor:LinftyOnHom}
	Let $\Psi:\Q\to\P$ be a morphism of dg operads, let $A$ be a $\P$-algebra, and let $D$ be a $\bar(\susp\otimes\Q)$-coalgebra. Then the chain complex $\hom(D,A)$ carries the structure of an $\L_\infty$-algebra by pulling back by the morphism $\genmaninDual_\Psi$. We denote the resulting $\L_\infty$-algebra by $\hom^\Psi(D,A)$.
\end{corollary}

\begin{proof}
	This follows immediately from Proposition \ref{prop:convolutionAlg} and Theorem \ref{thm:mainThmDual}.
\end{proof}

\subsection{The dual case}

Let $\P$ and $\Q$ be dg operads. If $\Q$ is reduced and finite dimensional in every arity, then so is $\bar(\susp\otimes\Q)$ and we have a natural isomorphism of operads
\[
\hom(\bar(\susp\otimes\Q),\P)\cong\P\otimes\bar(\susp\otimes\Q)^\vee\cong\P\otimes\Omega((\susp^{-1})^c\otimes\Q^\vee)\ .
\]
Under this correspondence, we get a morphism
\[
\genmanin_\Psi:\L_\infty\longrightarrow\P\otimes\Omega((\susp^{-1})^c\otimes\Q^\vee)\ .
\]
It can be described as follows. The data of a morphism of dg operads
\[
\Psi:\Q\longrightarrow\P
\]
is canonically equivalent to a collection of elements $\Psi_n\in\P(n)\otimes\Q(n)^\vee$ for $n\ge0$ satisfying certain conditions (given in Lemma \ref{lemma:technicalLemma}). Therefore, we can associate to $\Psi$ a map
\[
s^{-1}\lie^{\antishriek}\longrightarrow \P\otimes(s^{-1}(\susp^{-1})^c\otimes\Q^\vee)\subset\P\otimes\Omega((\susp^{-1})^c\otimes\Q^\vee)
\]
by sending $\ell_n$ to $s^{-1}\susp_n^{-1}\Psi_n$, and then commuting $s^{-1}\susp_n^{-1}$ with the part of $\Psi_n$ in $\P(n)$. Notice that as $\Psi$ is equivariant under the action of the symmetric group, the elements $\Psi_n$ are invariant under the action of $\S_n$. Thus the element we obtain with this map carries the sign representation. We denote by
\[
\genmanin_\Psi:\L_\infty\longrightarrow\P\otimes\Omega((\susp^{-1})^c\otimes\Q^\vee)
\]
the unique morphism of algebraic operads extending this map. Explicitly, it is given by taking an element of the operad $\L_\infty$, sending it to $\T(\P\otimes(s^{-1}(\susp^{-1})^c\otimes\Q^\vee))$ using the map given above at every vertex, then applying the morphism $\Phi$ described in Subsection \ref{subsect:freeOperads} to obtain an element of the tensor product $\T(\P)\otimes\T(s^{-1}(\susp^{-1})^c\otimes\Q^\vee)$, and finally applying the operadic composition map of $\P$ to the first part.

\medskip

In this dual setting, the main theorem (\ref{thm:mainThmDual}) becomes as follows.

\begin{theorem} \label{thm:mainThm}
	Let $\Psi:\Q\to\P$	be a morphism of dg operads such that $\Q(n)$ is finite dimensional for all $n\ge0$. Then the morphism of algebraic operads
	\[
	\genmanin_\Psi:\L_\infty\longrightarrow\P\otimes\Omega((\susp^{-1})^c\otimes\Q^\vee)
	\]
	described above commutes with the differentials, i.e. it is a morphism of dg operads. Moreover, it is compatible with compositions in the sense that, if we have a second morphism of dg operads
	\[
	\Theta:\R\longrightarrow\Q
	\]
	with $\R(n)$ finite dimensional for all $n\ge0$, then we get the following commutative diagram
	\begin{center}
		\begin{tikzpicture}
			\node (a) at (0,0){$\L_\infty$};
			\node (b) at (4,2){$\Q\otimes\Omega((\susp^{-1})^c\otimes\R^\vee)$};
			\node (c) at (4,0){$\P\otimes\Omega((\susp^{-1})^c\otimes\R^\vee)$};
			\node (d) at (4,-2){$\P\otimes\Omega((\susp^{-1})^c\otimes\Q^\vee)$};
			
			\draw[->] (a)--node[above]{$\genmanin_\Theta$}(b);
			\draw[->] (a)--node[above]{$\genmanin_{\Psi\Theta}$}(c);
			\draw[->] (a)--node[below]{$\genmanin_\Psi$}(d);
			
			\draw[->] (b)--node[right]{$\Psi\otimes1$}(c);
			\draw[->] (d)--node[right]{$1\otimes\Omega((\susp^{-1})^c\otimes\Theta^\vee)$}(c);
		\end{tikzpicture}
	\end{center}
	in the category of dg operads, that is:
	\[\genmanin_{\Psi\Theta} = (\Psi\otimes1)\genmanin_\Theta = (1\otimes\Omega((\susp^{-1})^c\otimes\Theta^\vee))\genmanin_\Psi\ .
	\]
\end{theorem}

\begin{corollary}
	Let $\Psi:\Q\to\P$ be a morphism of dg operads such that $\Q(n)$ is finite dimensional for all $n\ge0$, let $A$ be a $\P$-algebra, and let $C$ be a $\Omega((\susp^{-1})^c\otimes\Q^\vee)$-algebra. Then the tensor product $A\otimes C$ carries the structure of an $\L_\infty$-algebra given by pulling back by the morphism $\genmanin_\Psi$. We denote the tensor product $A\otimes C$ equipped with the induced $\L_\infty$-algebra structure by $A\otimes^\Psi C$. This is compatible with Corollary \ref{cor:LinftyOnHom} in the sense that, if $D$ is a finite dimensional $\bar(\susp\otimes\Q)$-coalgebra, then
	\[
	\hom^\Psi(D,A)\cong A\otimes^\Psi D^\vee
	\]
	as $\L_\infty$-algebras via the natural isomorphism.
\end{corollary}

Theorem \ref{thm:mainThm} can be proven by using Theorem \ref{thm:mainThmDual} and the natural isomorphism
\[
\hom(\bar(\susp\otimes\Q),\P)\cong\P\otimes\Omega((\susp^{-1})^c\otimes\Q^\vee)
\]
for $\Q$ reduced, but it can also directly be proven in this setting using the following lemma characterizing the sequence of elements $\Psi_n$ associated to $\Psi$, which will be used various times in the rest of the paper. This more direct proof allows us to get rid of the assumption that $\Q$ be reduced that we need for the above isomorphism to work.

\begin{lemma} \label{lemma:technicalLemma}
	Let $\P$ and $\Q$ be two dg operads such that $\Q(n)$ is finite dimensional in all arities. For each $n\ge0$, fix a homogeneous basis $q_1(n),\ldots,q_{m(n)}(n)$ of $\Q(n)$. A morphism of dg operads $\Psi:\Q\to\P$ is equivalent, by setting $p_i(n)\coloneqq\Psi(q_i(n))$, to a collection of $\S$-invariant elements
	\[
	\Psi_n \coloneqq \sum_{i=1}^{m(n)}p_i(n)\otimes q_i(n)^\vee\in\P(n)\otimes\Q^\vee(n)\ ,
	\]
	where $\P(n)\otimes\Q(n)^\vee$ is equipped with the diagonal action of the symmetric group, satisfying
	\begin{equation} \label{eqn:eqn1}
		\sum_{i=1}^{m(n)}\left(d_\P(p_i(n))\otimes q_i(n)^\vee + (-1)^{|p_i(n)|}p_i(n)\otimes d_\Q^\vee(q_i(n)^\vee)\right) = 0
	\end{equation}
	in $\P(n)\otimes\Q(n)^\vee$ and
	\begin{equation} \label{eqn:eqn2}
		\begin{aligned}
			\sum&{}_{i=1}^{m(n)}p_i(n)\otimes\Delta^{k,n_1,\ldots,n_k,\sigma}(q_i(n)^\vee) =\\
			=&\ \sum_{\substack{1\le i\le m(k)\\1\le i_j\le m(n_j)}}(-1)^\epsilon\gamma_\P\big(p_i(k)\otimes(p_{i_1}(n_1)\otimes\cdots\otimes p_{i_k}(n_k))^\sigma\big)\otimes\big(q_i^\vee(k)\otimes(q^\vee_{i_1}(n_1)\otimes\cdots\otimes q^\vee_{i_k}(n_k))^\sigma\big)\ ,
		\end{aligned}
	\end{equation}
	whenever $n_1+\cdots+n_k = n$ and $\sigma\in\Sh(n_1,\ldots,n_k)$, where
	\[
	\epsilon = |q_i(k)|\sum_{j=1}^k|q_{i_j}(n_j)| + \sum_{j=1}^k\sum_{j'>j}|q_{i_j}(n_j)||q_{i_{j'}}(n_{j'})|
	\]
	is the Koszul sign obtained by switching elements. This second equation holds in
	\[
	\P(n)\otimes(\Q\otimes_{k,n_1,\ldots,n_k}^\sigma\Q)^\vee\cong\P(n)\otimes(\Q^\vee\otimes_{k,n_1,\ldots,n_k}^\sigma\Q^\vee)\ .
	\]
	Notice that in this last isomorphism, signs may appear because of the Koszul convention.
\end{lemma}

From now on, we will ease notation by abstaining from indicating the arity of an element of the basis refers to, and write simply $q_i$ for $q_i(n)$. The correct arities can always easily be recovered from the context.

\begin{proof}
	Fix $n\ge0$. Notice that, as $\Psi$ has degree $0$, we have $|p_i| = |q_i|$ for all $i$.
	
	Invariance of the $\Psi_n$ under the $\S_n$-action is equivalent to the fact that $\Psi$ is equivariant. Let $q\in\Q(n)$ and let $\sigma\in\S_n$, then we have
	\[
	\sum_{i=1}^{m(n)}\langle q_i^\vee,q\rangle p_i = \Psi(q) = \Psi\left(q^{\sigma^{-1}}\right)^\sigma = \sum_{i=1}^{m(n)}\left\langle q_i^\vee,q^{\sigma^{-1}}\right\rangle p_i^\sigma = \sum_{i=1}^{m(n)}\left\langle (q_i^\vee)^\sigma,q\right\rangle p_i^\sigma\ ,
	\]
	and thus
	\[
	\sum_{i=1}^{m(n)}p_i^\sigma\otimes (q_i^\vee)^\sigma = \sum_{i=1}^{m(n)}p_i\otimes q_i^\vee\ .
	\]
	Equation (\ref{eqn:eqn1}) is nothing else than a restatement of the fact that $d_\P\Psi = \Psi d_\Q$. Indeed, let $q\in\Q(n)$, then
	\begin{align*}
		0 =&\ d_\P\Psi(q) -\Psi(d_\Q q)\\
		=&\ \sum_{i=1}^{m(n)}\left(\langle q_i^\vee,q\rangle d_\P p_i - \langle q_i^\vee,d_\Q q\rangle p_i\right)\\
		=&\ \sum_{i=1}^{m(n)}\left(\langle q_i^\vee,q\rangle d_\P p_i - (-1)^{|q_i|+1}\langle d_\Q^\vee q_i^\vee,q\rangle p_i\right)\\
		=&\ \sum_{i=0}^{m(n)}\left(d_\P(p_i)\otimes q_i^\vee + (-1)^{|p_i|}p_i\otimes d_\Q^\vee(q_i^\vee)\right)(q)\ .
	\end{align*}	
	Similarly, Equation (\ref{eqn:eqn2}) is equivalent to the fact that
	\[
	\gamma_\P(\Psi\circ\Psi) = \Psi\gamma_\Q\ .
	\]
	Let
	\[
	r\otimes(r_1,\ldots,r_k)^\sigma\in\Q\otimes_{k,n_1,\ldots,n_k}^\sigma\Q\ ,
	\]
	then
	\begin{align*}
		\Psi\big(\gamma_\Q\big(r\otimes(r_1,\ldots,r_k)^\sigma\big)\big) =&\ \sum_{i=1}^{m(n)}\langle q_i^\vee,\gamma_\Q(r\otimes(r_1,\ldots,r_k)^\sigma)\rangle p_i\\
		=&\ \sum_{i=1}^{m(n)}\langle\Delta^{k,n_1,\ldots,n_k,\sigma}(q_i^\vee),r\otimes(r_1,\ldots,r_k)^\sigma\rangle p_i\\
		=&\ \left(\sum_{i=1}^{m(n)}p_i\otimes\Delta^{k,n_1,\ldots,n_k,e}(q_i^\vee)\right)\big(r\otimes(r_1,\ldots,r_k)^\sigma\big)\ .
	\end{align*}
	On the other hand, we have
	\begin{align*}
		\gamma_\P&\left((\Psi\circ\Psi)\big(r\otimes(r_1\otimes\cdots\otimes r_k)^\sigma\big)\right) = \gamma_\P\big(\Psi(r)\otimes(\Psi(r_1),\ldots,\Psi(r_k))^\sigma\big)\\
		&= \gamma_\P\left(\sum_{\substack{1\le i\le m(k)\\1\le i_j\le m(n_j)}}\langle q_i^\vee,r\rangle\langle q_{i_1}^\vee,r_1\rangle\cdots\langle q_{i_k}^\vee,r_k\rangle p_i\otimes(p_{i_1}\otimes\cdots\otimes p_{i_k})^\sigma\right)\\
		&= \sum_{\substack{1\le i\le m(k)\\1\le i_j\le m(n_j)}}(-1)^\epsilon \langle q_i^\vee\otimes(q_{i_1}^\vee\otimes\cdots\otimes q_{i_k}^\vee)^\sigma,r\otimes(r_1,\ldots,r_k)^\sigma\rangle\gamma_\P\big(p_i\otimes(p_{i_1}\otimes\cdots\otimes p_{i_k})^\sigma\big)\\
		&= \left(\sum_{\substack{1\le i\le m(k)\\1\le i_j\le m(n_j)}}(-1)^\epsilon \gamma_\P\big(p_i\otimes(p_{i_1}\otimes\cdots\otimes p_{i_k})^\sigma\big)\otimes\big(q_i^\vee\otimes(q_{i_1}^\vee\otimes\cdots\otimes q_{i_k}^\vee)^\sigma\big)\right)(r\otimes(r_1,\ldots,r_k)^\sigma)\ ,
	\end{align*}
	completing the proof.
\end{proof}

\subsection{The non-symmetric case}

In the context of non-symmetric operads, our main theorem takes the following form.

\begin{theorem}
	Let $\Psi:\Q\to\P$ be a morphism of dg ns operads. Then there is a canonical morphism of dg ns operads
	\[
	\genmaninnsDual_\Psi:\A_\infty\longrightarrow\hom(\bar(\susp\otimes\Q),\P)\ .
	\]
	If moreover $\Q(n)$ is finite dimensional for all $n\ge0$, this can be dualized to give a canonical morphism of dg ns operads
	\[
	\genmaninns_\Psi:\A_\infty\longrightarrow\P\otimes\Omega((\susp^{-1})^c\otimes\Q^\vee)\ .
	\]
	Both are compatible with compositions, in the sense that if $\Theta:\R\to\Q$ is a second morphism of dg ns operads,
	\[
	\genmaninnsDual_{\Psi\Theta} = \Psi^*\genmaninnsDual_\Theta = \bar(\susp\otimes\Theta)_*\genmaninnsDual_\Psi\ ,
	\]
	and if moreover $\R(n)$ is finite dimensional for all $n\ge 2$, then
	\[
	\genmaninns_{\Psi\Theta} = (\Psi\otimes1)\genmaninns_\Theta = (1\otimes\Omega((\susp^{-1})^c\otimes\Theta^\vee))\genmaninns_\Psi\ .
	\]
\end{theorem}

The construction of the morphisms $\genmaninnsDual_\Psi$ and $\genmaninns_\Psi$ is analogous to the one for $\genmaninDual_\Psi$ and $\genmanin_\Psi$. At the level of algebras, this gives the following.

\begin{corollary}
	Let $\Psi:\Q\to\P$ be a morphism of dg ns operads, let $A$ be a $\P$-algebra, and let $D$ be a $\bar(\susp\otimes\Q)$-coalgebra. Then the chain complex of linear maps $\hom(D,A)$ carries a structure of $\A_\infty$-algebra by pullback by the morphism $\genmaninnsDual_\Psi$. We denote it again by $\hom^\Psi(D,A)$. If furthermore $\Q(n)$ is finite dimensional for all $n\ge0$ and $C$ is a $\Omega((\susp^{-1})^c\otimes\Q^\vee)$-algebra, then the tensor product $A\otimes C$ carries the structure of an $\A_\infty$-algebra induced by pullback by the morphism $\genmaninns_\Psi$. We denote it by $A\otimes^\Psi C$. Those two structures are compatible, in the sense that if $\Q$ is finite dimensional in all arities and $D$ is a finite dimensional $\bar(\susp\otimes\Q)$-coalgebra, then
	\[
	\hom^\Psi(D,A)\cong A\otimes D^\vee
	\]
	as $\A_\infty$-algebras via the natural isomorphism.
\end{corollary}

\section{Binary quadratic operads and Manin morphisms} \label{sect:sect3}

In this section, we restrict ourselves to binary quadratic operads and we explore the consequences of our main theorem in this context. We only work in the dual setting, i.e. with tensor products. In the present situation, the statement of Theorem \ref{thm:mainThm} simplifies to the existence of a map of dg operads
\[
\L_\infty\longrightarrow\P\otimes\Q^{\shriek}_\infty\ .
\]
We start by recalling the notion of what we call the Manin morphisms, which are morphisms arising from maps between operads via the adjunction between the black and white Manin products. We go on to prove that the morphisms obtained through our main theorem lift the Manin morphisms to a homotopical context. The principal result of this section lies in the fact that the morphisms obtained through Theorem \ref{thm:mainThm} are compatible with strict morphisms of algebras on one side and $\infty$-morphisms on the other side (Proposition \ref{prop:compatibilityMorphisms}). This amounts to the existence of a bifunctor given on objects by taking a $\P$-algebra and a $\Q^{\shriek}_\infty$-algebra and giving back their tensor product with the $\L_\infty$-algebra structure obtained through Theorem \ref{thm:mainThm} and allowing $\infty$-morphisms in the second slot.

\subsection{Manin morphisms}

In the category of operads given by binary quadratic data and morphisms induced by morphisms of quadratic data, one can define two operations, called the \emph{white} and \emph{black Manin products} and denoted by $\wmanin$ and $\bmanin$ respectively, both taking two binary quadratic operads and giving back another one. These objects first appeared in the context of algebras in \cite{manin87} and \cite{manin89} and then in \cite{ginzburg94} in relation to operads. For a more conceptual treatment, see \cite{vallette08} or \cite[Sect. 8.8]{vallette12}.

\begin{proposition}
	Fix a binary quadratic operad $\Q$. Then there is a natural isomorphism
	\[
	\hom_{\mathsf{bin.\ quad.\ op.}}(\R\bmanin\Q,\P)\cong\hom_{\mathsf{bin.\ quad.\ op.}}(\R,\P\wmanin\Q^{\shriek})\ .
	\]
	That is to say, the functors $-\bmanin\Q$ and $-\wmanin\Q^{\shriek}$ are adjoint. Moreover, the operad $\lie$ is a unit for the black product.
\end{proposition}

Therefore, any morphism
\[
\Psi:\lie\bmanin\Q\cong\Q\longrightarrow\P
\]
coming from a quadratic data is equivalent to a morphism
\[
\lie\longrightarrow\P\wmanin\Q^{\shriek}\ .
\]
As explained in \cite[Sect. 3.2]{vallette08}, the white product is the best binary quadratic approximation of the Hadamard product, and there is a canonical morphism
\[
\P\wmanin\Q^{\shriek}\longrightarrow\P\otimes\Q^{\shriek}\ .
\]

\begin{definition}
	We call the composite
	\[
	\manin_\Psi\coloneqq\big(\lie\longrightarrow\P\wmanin\Q^{\shriek}\longrightarrow\P\otimes\Q^{\shriek}\big)
	\]
	the \emph{Manin morphism} associated to $\Psi$.
\end{definition}

The Manin morphism $\manin_\Psi$ has the following explicit description. Assume $\P = \P(E,R)$ and $\Q = \P(F,S)$, fix a basis $f_1,\ldots,f_k$ of $F$, and let $e_1,\ldots,e_k\in E$ be the images of the $f_i$ under $\Psi$. Then $\manin_\Psi$ is the unique morphism of operads extending
\[
\manin_\Psi(b)\coloneqq\sum_{i=1}^ke_i\otimes s^{-1}\susp_2^{-1}f_i^\vee,
\]
where $b\in\lie(2)$ is the Lie bracket. Here, we implicitly used Lemma \ref{lemma:reprPshriek}.

\medskip

For any quadratic binary operad $\Q$, there is a canonical Manin morphism, namely the one associated to the identity of $\Q$, giving
\[
\manin_\Q\coloneqq\manin_{\id_\Q}:\lie\longrightarrow\Q\otimes\Q^!\ .
\]
It is easy to see that
\[
\manin_\Psi = (\Psi\otimes 1)\manin_\Q\ ,
\]
so that it is often only necessary to know $\manin_\Q$ to compute $\manin_\Psi$.

\subsection{Application of the main theorem to the binary quadratic case}

For the rest of this section, we fix two binary quadratic data $(E,R)$ and $(F,S)$ and denote by $\P=\P(E,R)$ and $\Q = \P(F,S)$ the two associated operads. Furthermore, we assume that $F$ is finite dimensional. We fix a morphism
\[
\Psi:\Q\longrightarrow\P
\]
in the category of binary quadratic operads. By Equation (\ref{eq:shriek-antishriek}) we have
\[
\Omega((\susp^{-1})^c\otimes\Q^\vee)\cong\Omega\big(\big(\Q^{\shriek}\big)^{\antishriek}\big)\cong\Q^{\shriek}_\infty\ .
\]
Thus, we can apply Theorem \ref{thm:mainThm} to the morphism $\Psi$ to obtain a map
\[
\genmanin_\Psi:\L_\infty\longrightarrow\P\otimes\Q^{\shriek}_\infty\ .
\]
Here we used the fact that since $F$ is finite dimensional, $\Q(n)$ is finite dimensional for all $n\ge0$. The following proposition shows that in this situation, Theorem \ref{thm:mainThm} provides us with a lifting of the Manin morphisms to operads coding homotopy algebras.

\begin{proposition} \label{prop:ManinGenmanin}
	The following square
	\begin{center}
		\begin{tikzpicture}
			\node (a) at (0,0){$\L_\infty$};
			\node (b) at (2.5,0){$\P\otimes\Q^{\shriek}_\infty$};
			\node (c) at (0,-1.5){$\lie$};
			\node (d) at (2.5,-1.5){$\P\otimes\Q^{\shriek}$};
			
			\draw[->] (a)--node[above]{$\genmanin_\Psi$}(b);
			\draw[->] (a)--(c);
			\draw[->] (b)--(d);
			\draw[->] (c)--node[above]{$\manin_\psi$}(d);
		\end{tikzpicture}
	\end{center}
	where the vertical maps are the canonical ones coming from the resolutions, is commutative.
\end{proposition}

\begin{proof}
	The left vertical arrow sends $\ell_2$ to $b$ and $\ell_n$ to $0$ for all $n\ge3$. Therefore, the south--west composite is the map sends
	\[
	\ell_2\longmapsto \sum_{i=1}^ke_i\otimes s^{-1}\susp^{-1}_2f_i^\vee
	\]
	and all the higher $\ell_n$ to zero. On the other hand, the right vertical arrow is given by tensoring the identity of $\P$ with the canonical resolution map
	\[
	\Q^{\shriek}_\infty\longrightarrow\Q^{\shriek}\ ,
	\]
	which is defined on the generators $s^{-1}(\Q^{\shriek})^{\antishriek}\cong s^{-1}(\susp^{-1})^c\otimes\Q^\vee$ as being the identity on arity $2$ and zero on all higher arities, this because $\Q$ is quadratic. By definition, the morphism $\genmanin_\Psi$ sends $\ell_n$ to an element of $\P(n)\otimes s^{-1}(\Q^{\shriek})^{\antishriek}(n)$. Therefore, the north--east composite gives zero on $\ell_n$, for $n\ge3$ and sends
	\[
	\ell_2\longmapsto \sum_{i=1}^ke_i\otimes s^{-1}\susp^{-1}_2f_i^\vee
	\]
	just like the other map.
\end{proof}

\subsection{Functoriality of the $\L_\infty$-algebra structure on tensor products}

The $\L_\infty$-algebra structure on the tensor product of a $\P$-algebra and a $\Q^{\shriek}_\infty$-algebra is functorial in $\P$-algebras. Indeed, given two $\P$-algebras $A$ and $A'$, a $\Q^{\shriek}_\infty$-algebra $C$, and a morphism
\[
f:A\longrightarrow A'
\]
of $\P$-algebras, there an obvious induced (strict) morphism of $\L_\infty$-algebras
\[
A\otimes^\Psi C\longrightarrow A'\otimes^\Psi C\ ,
\]
given by
\[
a\otimes c\longmapsto f(a)\otimes c\ .
\]
In fact, we have more than that: the functoriality also holds for $\Q^{\shriek}_\infty$-algebras in a strong way, that is with respect to $\infty$-morphisms. 

\medskip

Suppose we are given two $\P$-algebras $A$ and $A'$, as well as two $\Q^{\shriek}_\infty$-algebras $C$ and $C'$, a morphism of $\P$-algebras $f:A\to A'$, and an $\infty$-morphism of $\Q^{\shriek}_\infty$-algebras $g:C\rightsquigarrow C'$. Then we can use these objects to define a morphism of $\lie^{\antishriek}$-coalgebras, i.e. suspended cocommutative coalgebras,
\[
f\otimes^\Psi g:\lie^{\antishriek}(A\otimes C)\longrightarrow\lie^{\antishriek}(A'\otimes C')
\]
as the unique morphism of suspended cofree cocommutative coalgebras extending the map sending
\[
\susp_n^{-1}\mu_n^\vee\otimes(a_1\otimes c_1)\otimes\cdots\otimes(a_n\otimes c_n)\in\lie^{\antishriek}(n)\otimes(A\otimes C)^{\otimes n}
\]
to
\[
\sum_{i=1}^{m(n)}(-1)^{(n-1)|p_i| + \epsilon}f\big(\rho_A(p_i)(a_1,\ldots,a_n)\big)\otimes g_n\big(\susp_n^{-1}q_i^\vee\otimes c_1\otimes\cdots\otimes c_n\big)\ ,
\]
where
\[
\epsilon = \sum_{i=1}^n\sum_{j<i}|a_i||c_j| + (n-1)\sum_{i=1}^n|a_i|
\]
is the sign obtained via Koszul rule by reordering terms, and $\rho_A$ is the $\P$-algebra structure of $A$ seen as a map
\[
\rho_A:\P\longrightarrow\End_A\ .
\]
This morphism can be seen as first sending $\susp_n^{-1}\mu_n^\vee$ to $\P(n)\otimes(\Q^{\shriek})^{\antishriek}$ using $s\genmanin_\Psi s^{-1}$, rearranging terms to get elements of $(\P(n)\otimes_{\S_n} A^{\otimes n})\otimes((\Q^{\shriek})^{\antishriek}(n)\otimes_{\S_n} C^{\otimes n})$, and then applying $f$ and the composition of $A$ on the first part and $g_n$ on the second part.

\begin{proposition} \label{prop:compatibilityMorphisms}
	Let $A,A'$ be $\P$-algebras, and let $C,C'$ be $\Q^{\shriek}_\infty$-algebras. Further, let $f:A\to A'$ be a morphism of $\P$-algebras, and $g:C\rightsquigarrow C'$ be an $\infty$-morphism of $\Q^{\shriek}_\infty$-algebras. The map defined above is a morphism
	\[
	f\otimes^\Psi g:\bar_\iota(A\otimes^\Psi C)\longrightarrow\bar_\iota(A'\otimes^\Psi C')
	\]
	of dg suspended cocommutative coalgebras, i.e. an $\infty$-morphism of $\L_\infty$-algebras. Moreover, if $f':A'\to A''$ is a second morphism of $\P$-algebras and $g':C'\rightsquigarrow C''$ is another $\infty$-morphism of $\Q^{\shriek}_\infty$-algebras, then
	\[
	(f'\circ f)\otimes^\Psi(g'\circ g) = (f'\otimes^\Psi g')\circ(f\otimes^\Psi g)\ .
	\]
\end{proposition}

\begin{proof}
	For this proof, let $\rho_A$ be the algebraic structure of $A$ seen as a map
	\[
	\rho_A:\P\longmapsto\End_A
	\]
	and let $\varphi_C$ be the algebraic structure on $C$, but seen as a twisting morphism
	\[
	\varphi_C\in\Tw((\Q^{\shriek})^{\antishriek},\End_C)\cong\Tw((\susp^{-1})^c\otimes\Q^\vee,\End_C)\ ,
	\]
	and similarly for $A'$ and $C'$. The algebraic structure of $A\otimes^\Psi C$ is also seen as a twisting morphism
	\[
	\varphi_{A\otimes C}\in\Tw(\lie^{\antishriek},\End_{A\otimes C})\cong\Tw((\susp^{-1})^c\otimes\com^\vee,\End_{A\otimes C})
	\]
	and the same is true for the one of $A'\otimes^\Psi C'$. We denote by $\delta$ the differentials of both $\bar_\iota(A\otimes^\Psi C)$ and $\bar_\iota(A'\otimes^\Psi C')$.
	
	To ease notation, we adopt the following conventions:
	\begin{itemize}
		\item We usually do not write the elements of the algebras, leaving them implicit. For example, for the element
		\[
		\susp_n^{-1}\mu_n^\vee\otimes(a_1\otimes c_1)\otimes\cdots\otimes(a_n\otimes c_n)\in\lie^{\antishriek}\otimes(A\otimes C)^{\otimes n}
		\]
		we write just $\susp_n^{-1}\mu_n^\vee$.	The reordering of the elements of the algebras are also left implicit, so that $\susp_n^{-1}\mu_n^\vee$ also denotes the element
		\[
		(-1)^\epsilon\susp_n^{-1}\mu_n^\vee\otimes(a_1\otimes\cdots\otimes a_n)\otimes(c_1\otimes\cdots\otimes c_n)\ ,
		\]
		where $\epsilon$ is the appropriate Koszul sign. Notice that, with this convention, we have for example that $d_{(A\otimes C)^{\otimes n}} = d_{A^{\otimes n}\otimes C^{\otimes n}}$.
		\item Whenever we have an element of $\End_A$, $\End_C$, or $\End_{A\otimes C}$, we implicitly apply it to the elements of the respective algebra. For example, we write
		\[
		(1\otimes_j\rho_A)(p_1\otimes_j p_2)
		\]
		with $p_1\in\P(n_1)$ and $p_2\in\P(n_2)$, with $n_1 + n_2 = n + 1$, for
		\[
		(-1)^{|p_2|\sum_{i=1}^{j-1}|a_i|}p_1\otimes(a_1\otimes\cdots\otimes a_{j-1}\otimes\rho(p_2)(a_j\otimes\cdots\otimes a_{j+n_2})\otimes a_{j+n_2+1}\otimes\cdots\otimes a_n)\ .
		\]
		\item We use the fact that
		\[
		\Delta^\tau(\susp_n^{-1}\mu_n^\vee) = T^\tau(\Delta^\tau(\susp_n^{-1})\Delta^\tau(\mu_n^\vee))
		\]
		to avoid writing unnecessary signs. We abuse of notation and also use $T^\tau$ to identify
		\[
		\big((\susp^{-1})^c\big)^\tau\otimes\T(\P)^\tau\otimes\T(\Q^\vee)^\tau
		\]
		with
		\[
		\T(\P)^\tau\otimes\T((\susp^{-1})^c\otimes\Q^\vee)
		\]
		and so on. To ease notation, we also just write $T$ for $T^\tau$, as the underlying tree can always easily be recovered from the context.
	\end{itemize}
	
	We start by proving that $f\otimes^\Psi g$ is an $\infty$-morphism of $\L_\infty$-algebras. We consider
	\[
	\susp_n^{-1}\mu_n^\vee\otimes(a_1\otimes c_1)\otimes\cdots\otimes(a_n\otimes c_n)\in\lie^{\antishriek}\otimes(A\otimes C)^{\otimes n}.
	\]
	We have
	\begin{align*}
		\delta&\left(\susp_n^{-1}\mu_n^\vee\otimes(a_1\otimes c_1)\otimes\cdots\otimes(a_n\otimes c_n)\right) =\\
		& \quad= (-1)^{n-1}\susp_n^{-1}\mu_n^\vee\otimes d_{(A\otimes C)^{\otimes n}} + \sum_{\substack{n_1+n_2 = n+1\\1\le j\le n_1\\\sigma\in\Sh(n_1-1,n_2)}}\left(1\otimes_j\varphi_{A\otimes C}\right)T\left(\Delta^{n_1,n_2,j,\sigma}\left(\susp_n^{-1}\right)(\mu_{n_1}^\vee\otimes_j\mu_{n_2}^\vee)^\sigma\right)\ .
	\end{align*}
	Writing out explicitly $\varphi_{A\otimes C}$, the second term is
	\[
	\sum_{\substack{n_1+n_2 = n+1\\1\le j\le n_1\\\sigma\in\Sh(n_1-1,n_2)}}\left(1\otimes_j(\rho_A\otimes\varphi_C)\right)T\left(\Delta^{n_1,n_2,j,\sigma}\left(\susp_n^{-1}\right)\sum_{1\le i_2\le m(n_2)}(\mu_{n_1}^\vee\otimes_j(p_{i_2}\otimes q_{i_2}^\vee))^\sigma\right)\ .
	\]
	Now we apply $f\otimes^\Psi g$ and then project onto $A'\otimes C'$ to obtain
	\begin{align*}
		(\text{proj}\circ&(f\otimes^\Psi g)\circ\delta)(\susp_n^{-1}\mu_n^\vee) =\\
		=&\ \sum_{i=1}^{m(n)}(-1)^{(n-1)(|p_i| + 1))}(f\rho_A\otimes g_n)(p_i\otimes\susp_n^{-1}q_i^\vee)d_{A^{\otimes n}\otimes C^{\otimes n}}\\
		& + \sum_{\substack{n_1+n_2 = n+1\\1\le j\le n_1\\ \sigma\in\Sh(n_1-1,n_2)}}(f\rho_A\otimes(g_{n_1}\otimes_j\varphi_C))\otimes\\
		&\qquad\qquad\otimes T\left(\Delta^{n_1,n_2,j,\sigma}(\susp_n^{-1})\sum_{\substack{1\le i_1\le m(n_1)\\1\le i_2\le m(n_2)}}(-1)^{|p_{i_2}||q_{i_1}|}\gamma_\P(p_{i_1}\otimes_jp_{i_2})^\sigma\otimes(q_{i_1}^\vee\otimes_jq_{i_2}^\vee)^\sigma\right)\ .
	\end{align*}
	Then, using Lemma \ref{lemma:technicalLemma} (\ref{eqn:eqn2}) and switching terms, remembering that $|\varphi_C| = -1$, this equals
	\begin{align*}
		\sum_{i=1}^{m(n)}&(-1)^{(n-1)(|p_i| + 1)) + (n - 1 + |q_i|)}f\rho_A(p_i)d_{A^{\otimes n}}\otimes g_n(\susp_n^{-1}q_i^\vee)\\
		& + \sum_{i=1}^{m(n)}(-1)^{(n-1)(|p_i| + 1))}f\rho_A(p_i)\otimes g_n(\susp_n^{-1}q_i^\vee)d_{c^{\otimes n}}\\
		& + \sum_{i=1}^{m(n)}(-1)^{n|p_i|}f\rho_A(p_i)\otimes\left(\sum_{\substack{n_1+n_2 = n+1\\1\le j\le n_1\\ \sigma\in\Sh(n_1-1,n_2)}}(g_{n_1}\otimes_j\varphi_C)\Delta^{n_1,n_2,j,\sigma}(\susp_n^{-1}q_i)\right)\\
		=&\ \sum_{i=1}^{m(n)}(-1)^{n|p_i|}f\rho_A(p_i)d_{A^{\otimes n}}\otimes g_n(\susp_n^{-1}q_i^\vee) + \sum_{i=1}^{m(n)}(-1)^{n|p_i|}f\rho_A(p_i)\otimes\left(\text{proj}\circ g\circ d_{\bar_\iota(C)}(\susp_n^{-1}q_i^\vee)\right)\ ,
	\end{align*}
	where we used the fact that $|p_i| = |q_i|$. On the other hand,
	\begin{align*}
		&(f\otimes^\Psi g)\big(\susp_n^{-1}\mu_n^\vee\otimes(a_1\otimes c_1)\otimes\cdots\otimes(a_n\otimes c_n)\big) =\\
		&= \sum_{\substack{k\ge1\\n_1+\ldots+n_k=n\\\sigma\in\Sh(n_1,\ldots,n_k)}}(1\circ((f\rho_A)^{\otimes k}\otimes(g_{n_1}\otimes\cdots\otimes g_{n_k}))\otimes\\
		&\qquad\quad\otimes T\left(\Delta^{k,n_1,\ldots,n_k,\sigma}(\susp_n^{-1})\sum_{\substack{1\le h\le m(k)\\1\le i_j\le m(n_j)}}(-1)^{\epsilon'}(p_h\otimes(p_{i_1}\otimes\cdots\otimes p_{i_k}))^\sigma\otimes(q_h^\vee\otimes(q^\vee_{i_1}\otimes\cdots\otimes q^\vee_{i_k}))^\sigma\right)\ ,
	\end{align*}
	where $\epsilon'$ is the sign appearing by rearranging the $p_{i_j}$'s and $q_{i_j}$'s. Now we apply $\delta$ and project onto $A'\otimes C'$ to get
	\begin{align*}
		d&{}_{A'\otimes C'}(f\rho_A\otimes g_n)\left(\sum_{i=1}^{m(n)}(-1)^{(n-1)|p_i|}p_i\otimes\susp_n^{-1}q_i^\vee\right)\\
		&\quad+ \sum_{\substack{k\ge1\\n_1+\ldots+n_k=n\\\sigma\in\Sh(n_1,\ldots,n_k)}}(\rho_{A'}\circ(f\rho_A)^{\otimes k}\otimes \varphi_{C'}\circ(g_{n_1}\otimes\cdots\otimes g_{n_k}))\otimes\\
		&\qquad\quad\otimes T\left(\Delta^{k,n_1,\ldots,n_k,\sigma}(\susp_n^{-1})\sum_{\substack{1\le h\le m(k)\\1\le i_j\le m(n_j)}}(-1)^{\epsilon'}(p_h\otimes(p_{i_1}\otimes\cdots\otimes p_{i_k}))^\sigma\otimes(q_h^\vee\otimes(q^\vee_{i_1}\otimes\cdots\otimes q^\vee_{i_k}))^\sigma\right)\\
		&= \sum_{i=1}^{m(n)}(-1)^{(n-1)|p_i|}d_{A'\otimes C'}\big(f\rho_A(p_i)\otimes g_n(\susp_n^{-1})\big)\\
		&\quad+ \sum_{\substack{k\ge1\\n_1+\ldots+n_k=n\\\sigma\in\Sh(n_1,\ldots,n_k)}}f\rho_A\otimes(\varphi_{C'}\otimes(g_{n_1}\otimes\cdots\otimes g_{n_k}))\otimes\\
		&\qquad\quad\otimes T\left(\Delta^{k,n_1,\ldots,n_k,\sigma}(\susp_n^{-1})\sum_{\substack{1\le h\le m(k)\\1\le i_j\le m(n_j)}}(-1)^{\epsilon'}(p_h\circ(p_{i_1},\ldots,p_{i_k}))^\sigma\otimes(q_h^\vee\otimes(q^\vee_{i_1}\otimes\cdots\otimes q^\vee_{i_k}))^\sigma\right)\ ,
	\end{align*}
	where we used the fact that $f$ is a strict morphism of $\P$-algebras. By Lemma \ref{lemma:technicalLemma} (\ref{eqn:eqn2}) and a switch, the last term equals equals
	\[
	\sum_{\substack{k\ge1\\n_1+\ldots+n_k=n\\\sigma\in\Sh(n_1,\ldots,n_k)}}(-1)^{(n-1)|p_i|}(f\rho_A)\otimes(\varphi_{C'}\otimes(g_{n_1}\otimes\cdots\otimes g_{n_k}))\otimes\sum_{1=1}^{m(n)}p_i\otimes\Delta^{k,n_1,\ldots,n_k,\sigma}(\susp_n^{-1}q_i^\vee)\ ,
	\]
	and so the whole expression evaluates to
	\begin{align*}
		\sum_{i=1}^{m(n)}&(-1)^{(n-1)|p_i|}d_{A'}f\rho_A(p_i)\otimes g_n(\susp_n^{-1}q_i^\vee) + \sum_{i=1}^{m(n)}(-1)^{n|p_i|}f\rho_A(p_i)\otimes d_{C'}g_n(\susp_n^{-1}q_i^\vee)\\
		& + \sum_{\substack{k\ge1\\n_1+\ldots+n_k=n\\\sigma\in\Sh(n_1,\ldots,n_k)}}(-1)^{n|p_i|} f\rho_A(p_i)\otimes\left(\sum_{i=1}^{m(n)}(\varphi_{C'}\otimes(g_{n_1}\otimes\cdots\otimes g_{n_k}))\right)\Delta^{k,n_1,\ldots,n_k,\sigma}(\susp_n^{-1}q_i^\vee)\\
		=&\ \sum_{i=1}^{m(n)}(-1)^{(n-1)|p_i|}d_{A'}f\rho_A(p_i)\otimes g_n(\susp_n^{-1}q_i^\vee)\\
		& + \sum_{\substack{k\ge1\\n_1+\ldots+n_k=n\\\sigma\in\Sh(n_1,\ldots,n_k)}}(-1)^{n|p_i|}f\rho_A(p_i)\otimes\big(\text{proj}\circ d_{\bar_\iota(C')}\circ g(\susp_n^{-1}q_i^\vee)\big)\ .
	\end{align*}
	As $g$ is an $\infty$-morphism of $\Q^{\shriek}_\infty$-algebras, comparison of the term we evaluated above implies
	\[
	(f\otimes^\Psi g)\circ\delta = \delta\circ(f\otimes^\Psi g)\ ,
	\]
	so that $(f\otimes^\Psi g)$ is indeed an $\infty$-morphism of $\L_\infty$-algebras.
	
	The proof of the fact that this assignment respects compositions can be proven in a similar way, and is left as an exercise to the dedicated reader.
\end{proof}

This result amounts to the following statement.

\begin{theorem} \label{thm:bifunctor}
	There is a bifunctor
	\[
	\otimes^\Psi:\P\text{-}\mathsf{alg}\times\infty\text{-}\Q^{\shriek}_\infty\text{-}\mathsf{alg}\longrightarrow\infty\text{-}\L_\infty\text{-}\mathsf{alg}
	\]
	taking a $\P$-algebra $A$ and a $\Q^{\shriek}_\infty$-algebra $C$ and giving back $A\otimes^\Psi C$. The action on maps is given by Proposition \ref{prop:compatibilityMorphisms}.
\end{theorem}

\begin{remark}
	One can write an analogous version of the results presented in this section and the next one for the homotopy algebra structures obtained on spaces of linear maps through Theorem \ref{thm:mainThmDual}. The reason we have not done so is that the notion of $\infty$-morphisms and the Homotopy Transfer Theorem for homotopy coalgebras --- which are expected to work in a similar way to the notion for algebras --- have never been developed in the literature, and it is not in the scope of the present article to do it.
\end{remark}

\subsection{The non-symmetric case}

There is a theory of Manin products also in the ns case, see \cite[Sect. 8.8.9]{vallette12}, with $\as$ taking the role of $\lie$ as unit for the black product. As above, we can use the adjunction between the products to associate to a morphism $\Psi:\Q\to\P$ of ns binary quadratic operads (coming from a morphism of the underlying quadratic data) a morphism
\[
\maninns_\Psi:\as\longrightarrow\P\otimes\Q^{\shriek}\ .
\]

\begin{proposition}
	The following square
	\begin{center}
		\begin{tikzpicture}
			\node (a) at (0,0){$\A_\infty$};
			\node (b) at (2.5,0){$\P\otimes\Q^{\shriek}_\infty$};
			\node (c) at (0,-1.5){$\as$};
			\node (d) at (2.5,-1.5){$\P\otimes\Q^{\shriek}$};
			
			\draw[->] (a)--node[above]{$\genmaninns_\Psi$}(b);
			\draw[->] (a)--(c);
			\draw[->] (b)--(d);
			\draw[->] (c)--node[above]{$\maninns_\psi$}(d);
		\end{tikzpicture}
	\end{center}
	where the vertical maps are the canonical ones coming from the resolutions, is commutative.
\end{proposition}

The result analogous to the compatibility of $\genmanin_\Psi$ with various notions of morphisms into play takes the following form.

\begin{proposition} \label{propNS:compatibilityMorphisms}
	Let $\Psi:\Q\to\P$ be a morphism of dg ns binary quadratic operads such that $\Q$ is finitely generated. Let $A,A'$ be $\P$-algebras and let $C,C'$ be $\Q^{\shriek}_\infty$-algebras. Further, let $f:A\to A'$ be a morphism of $\P$-algebras and $g:C\rightsquigarrow C'$ be an $\infty$-morphism of $\Q^{\shriek}_\infty$-algebras. Then there is a canonical $\infty$-morphism of $\A_\infty$-algebras
	\[
	f\otimes^\Psi g:\bar_\iota(A\otimes^\Psi C)\longrightarrow\bar_\iota(A'\otimes^\Psi C')\ ,
	\]
	constructed analogously to the symmetric case. Moreover, if $f':A'\to A''$ is a second morphism of $\P$-algebras and $g':C'\rightsquigarrow C''$ is another $\infty$-morphism of $\Q^{\shriek}_\infty$-algebras, then
	\[
	(f'\circ f)\otimes^\Psi(g'\circ g) = (f'\otimes^\Psi g')\circ(f\otimes^\Psi g)\ .
	\]
\end{proposition}

Again as before, this amounts to the following statement.

\begin{theorem} \label{thmNS:bifunctor}
	There is a bifunctor
	\[
	\otimes^\Psi:\P\text{-}\mathsf{alg}\times\infty\text{-}\Q^{\shriek}_\infty\text{-}\mathsf{alg}\longrightarrow\infty\text{-}\A_\infty\text{-}\mathsf{alg}
	\]
	taking a $\P$-algebra $A$ and a $\Q^{\shriek}_\infty$-algebra $C$ and giving back $A\otimes^\Psi C$, with the action on maps being given by Proposition \ref{propNS:compatibilityMorphisms}.
\end{theorem}

\section{Compatibility with the Homotopy Transfer Theorem} \label{sect:sect4}

In this section, we stay in the context of binary Koszul operads. Again, we only work in the dual scenario. Recall the Homotopy Transfer Theorem, which tells us that if we are given any homotopy algebra whose underlying chain complex retracts to another chain complex, then we can induce a natural homotopy algebra structure on the second chain complex without losing any homotopical information. In particular, this is true when the starting homotopy algebra is strict. Suppose we are given two algebras over two operads related by a morphism as in the statement of the main theorem (\ref{thm:mainThm}). Consider a retraction of the underlying chain complex of the second algebra. There are two ways to induce a homotopy Lie algebra structure on the tensor product of the first chain complex and the retracted one: we can either first use the Manin morphism and then the Homotopy Transfer Theorem with the induced retraction on the tensor product, or we can first use the Homotopy Transfer Theorem with the original retraction and then apply the main theorem. We prove that the two structures obtained in this way are equal.

\subsection{Two ways to obtain a homotopy Lie algebra structure}

Suppose $\Q$ is a binary Koszul operad, let $A$ be a $\P$-algebra and let $B$ be a $\Q^{\shriek}$-algebra. Suppose further that we have a retraction of chain complexes
\begin{center}
	\begin{tikzpicture}
		\node (a) at (0,0){$B$};
		\node (b) at (2,0){$C$};
		
		\draw[->] (a)++(.3,.1)--node[above]{\mbox{\tiny{$p$}}}+(1.4,0);
		\draw[<-,yshift=-1mm] (a)++(.3,-.1)--node[below]{\mbox{\tiny{$i$}}}+(1.4,0);
		\draw[->] (a) to [out=-150,in=150,looseness=4] node[left]{\mbox{\tiny{$h$}}} (a);
	\end{tikzpicture}
\end{center}
from $B$ to $C$. Then we have the following two natural ways to obtain an $\L_\infty$-algebra structure on the tensor product $A\otimes C$.
\begin{enumerate}
	\item Pull back the $(\P\otimes\Q^{\shriek})$-algebra structure on $A\otimes B$ to a Lie algebra structure using the Manin morphism $\manin_\Psi$, then transfer this structure to an $\L_\infty$-algebra structure $\{\ell_n\}_{n\ge2}$ on $A\otimes C$ using the retraction
	\begin{center}
		\begin{tikzpicture}
			\node (a) at (0,0){$A\otimes B$};
			\node (b) at (3,0){$A\otimes C\ .$};
			
			\draw[->] (a)++(.7,.1)--node[above]{\mbox{\tiny{$1\otimes p$}}}+(1.6,0);
			\draw[<-,yshift=-1mm] (a)++(.7,-.1)--node[below]{\mbox{\tiny{$1\otimes i$}}}+(1.6,0);
			\draw[->] (a) to [out=-150,in=150,looseness=4] node[left]{\mbox{\tiny{$1\otimes h$}}} (a);
		\end{tikzpicture}
	\end{center}
	\item Use the retraction to transfer the $\Q^{\shriek}$-algebra structure on $B$ to a $\Q^{\shriek}_\infty$-algebra structure on $C$, then pull back the $(\P\otimes\Q^{\shriek}_\infty)$-algebra structure on $A\otimes C$ to an $\L_\infty$-algebra structure $\{\tilde{\ell}_n\}_{n\ge2}$ using the morphism $\genmanin_\Psi$.
\end{enumerate}

\subsection{The two structures are equal}

\begin{theorem} \label{thm:twoLinftyStrAreEqual}
	Let $\Psi:\Q\to\P$ be a morphism of binary quadratic dg operads such that $\Q$ is finitely generated and Koszul. Then the two $\L_\infty$-algebra structures obtained on $A\otimes C$ as described above are equal. Moreover, the respective $\infty$-morphisms extending the morphisms of the retraction are related by
	\[
	(1\otimes i)_\infty = 1\otimes^\Psi i_\infty\qquad\text{and}\qquad(1\otimes p)_\infty = 1\otimes^\Psi p_\infty\ .
	\]
\end{theorem}

\begin{proof}
	We begin by establishing a bit of notation. We denote by $\brt$ the set of binary rooted trees, and by $\brt_n$ the subset given by trees of arity $n$. We fix a basis $f_1,\ldots,f_k$ of $F$ and let $e_i\coloneqq\Psi(f_i)$. For all $\tau\in\brt$, we fix once and for all a numbering of the vertices. For $\tau\in\brt_{n+1}$ and binary operations $\alpha_1,\ldots,\alpha_n$, we denote by $\tau(\alpha_1,\ldots,\alpha_n)$ the tree with the $\alpha_i$ inserted at the vertex $i$. Let $\P = \P(E,R)$ and $\Q = \P(F,S)$ be the quadratic presentations of $\P$ and $\Q$. We write
	\[
	\tilde{\gamma}_\P:\T(E)\longrightarrow\P = \T(E)/(R)
	\]
	for the composition map (i.e. the quotient map), and similarly for $\Q$. Notice that this notation is consistent, because we can see $\T(E)$ as a subspace of $\T(\P)$ since $E = \P(2)$, and there the map $\tilde{\gamma}_\P$ is exactly the usual operadic composition map. We represent the $\P$-algebra structure on $A$ by the degree $0$ map of operads
	\[
	\rho_A:\P\longrightarrow\End_A
	\]
	and similarly for $A'$, while all other algebra structures are given by twisting morphisms and denoted by $\varphi_B$, $\varphi_{A\otimes B}$ and so on. Notice that as $B$ is also a strict $\Q^{\shriek}$ algebra, and that $\Q^{\shriek}$ is a binary operad, we have that
	\[
	\varphi_B\in\Tw((\Q^{\shriek})^{\antishriek},\End_B)\cong\Tw((\susp^{-1})^c\otimes\Q^\vee,\End_B)
	\]
	splits through $(\Q^{\shriek})^{\antishriek}(2)\cong\susp_2^{-1}F^\vee$. Finally, let
	\[
	\vdw_C:\T^c(s\End_B)\longrightarrow\End_C
	\]
	be the Van der Laan morphism associated to the retraction from $B$ to $C$, and similarly denote by $\vdw_{A\otimes C}$ the Van der Laan morphism associated to the induced retraction from $A\otimes B$ to $A\otimes C$.
	
	\medskip
	
	We start by unwinding the definitions to give an explicit formulation for the second structure. For $n\ge2$, fix a basis $q_1,\ldots,q_{m(n)}$ of $\Q(n)$. As we have $(\Q^{\shriek})^{\antishriek} \cong (\susp^{-1})^c\otimes\Q^\vee$, this also gives us the basis $\{\susp^{-1}_nq_i^\vee\}_i$ of $(\Q^{\shriek})^{\antishriek}(n)$. By the Homotopy Transfer Theorem, the $\Q^{\shriek}_\infty$-algebra structure on $C$ is given by the twisting morphism
	\[
	\varphi_C\in\Tw((\Q^{\shriek})^{\antishriek},\End_C)
	\]
	equal to the composite
	\[
	(\Q^{\shriek})^{\antishriek}\xrightarrow{\tilde{\Delta}_{(\Q^{\shriek})^{\antishriek}}}\T^c\left(\overline{(\Q^{\shriek})^{\antishriek}}\right)\xrightarrow{\T^c(s\varphi_B)}\T^c(s\End_B)\xrightarrow{\vdw_C}\End_C\ .
	\]
	Since the chain complex $B$ has the structure of a strict $\Q^{\shriek}$-algebra, and since $\Q^{\shriek}$ is a binary operad, the second arrow gives zero on all non-binary trees. Therefore, we can rewrite the composite above as
	\[
	(\Q^{\shriek})^{\antishriek}\xrightarrow{\tilde{\Delta}_{(\Q^{\shriek})^{\antishriek}}^{\text{bin}}}\T^c(\susp_2^{-1}F^\vee)\xrightarrow{\T^c(s\varphi_B)}\T^c(s\End_B)\xrightarrow{\vdw_C}\End_C\ ,
	\]
	where $\tilde{\Delta}_{(\Q^{\shriek})^{\antishriek}}^{\text{bin}}$ is the composite of $\tilde{\Delta}_{(\Q^{\shriek})^{\antishriek}}$ with the projection on binary trees. Notice that it is dual to the composition map
	\[
	\tilde{\gamma}_{\susp\otimes\Q}:\T(\susp_2F)\longrightarrow\susp\otimes\Q\ .
	\]
	Putting all of this together, we get
	\[\varphi_C = \vdw_C\T^c(s\varphi_B)\tilde{\Delta}_{(\Q^{\shriek})^{\antishriek}}^{\text{bin}}\ .
	\]
	Using this, we obtain for the second $\L_\infty$-algebra structure:
	\begin{align*}
		\tilde{\ell}_n \coloneqq&\ (\rho_A\Psi\otimes\varphi_C)\left(\sum_{i=1}^{m(n)}(-1)^{|q_i|(n-1)}q_i\otimes \susp_n^{-1}q_i^\vee\right)\\
		=&\ \big(\rho_A\Psi\otimes(\vdw_C\T^c(s\varphi_B))\big)\left(\sum_{\tau\in\brt_n}\sum_{i=1}^{m(n)}(-1)^{|q_i|(n-1)}q_i\otimes\tilde{\Delta}^\tau(\susp_n^{-1}q_i^\vee)\right)\\
		=&\ \big(\rho_A\Psi\otimes(\vdw_C\T^c(s\varphi_B))\big)\left(\sum_{\tau\in\brt_n}T\left(\Delta^\tau(\susp_n^{-1})\sum_{i=1}^{m(n)}q_i\otimes\tilde{\Delta}^\tau(q_i^\vee)\right)\right)\ ,
	\end{align*}
	where the correct switches are left implicit. In the second line, the map $\tilde{\Delta}^\tau$ is the dual to the monadic composition map $\T(\Q)\to\Q$ restricted to the subspace $\T(\Q)^\tau$ with underlying tree $\tau$.
	
	\medskip
	
	Now we make explicit the first $\L_\infty$-algebra structure. The Homotopy Transfer Theorem gives
	\[
	\ell_n \coloneqq \vdw_{A\otimes C}\T^c(s\varphi_{A\otimes B})\tilde{\Delta}_{\lie^{\antishriek}}^{\text{bin}}(\susp_n^{-1}\mu_n^\vee)\ .
	\]
	Since the retraction of $A\otimes B$ to $A\otimes C$ is given by tensoring the maps of the retraction of $B$ to $C$ with the identity on $A$, the Van der Laan map $\vdw_{A\otimes C}$ can be computed by first using $\Phi$ to double the tree in $\T^c(s\End_{A\otimes B})$, obtaining an element of $\T^c(\End_A)\otimes\T^c(s\End_B)$, then composing the operations at the vertices of the first tree and applying the Van der Laan morphism $\vdw_C$ to the second tree. Writing this out, we have
	\[
	\vdw_{A\otimes C} = (\gamma_{\End_A}\otimes\vdw_C)\Phi\ .
	\]
	Now, the map $\T^c(s\varphi_{A\otimes B})$ is given by applying $\manin_\Psi$ at each vertex, and then applying $\rho_A\otimes s\varphi_B$ at each vertex. So it is straightforward to see that we have
	\[
	\Phi\T^c(s\varphi_{A\otimes B}) = \big(\T^c(\rho_A)\otimes\T^c(s\varphi_B)\big)\Phi\T^c(\manin_\Psi)\ ,
	\]
	that is to say, if we apply $\T^c(s\varphi_{A\otimes B})$ and then double the tree, we obtain the same thing that we get by applying $\manin_\Psi$ at each vertex, then immediately doubling the tree, and only then applying $\rho_A$ to the vertices of the first copy of the tree and $s\varphi_C$ to the vertices of the second one. Now we use the fact that $\manin_\Psi = (\Psi\otimes1)\manin_\Q$ and commute the $(\Psi\otimes1)$ over the map $\Phi$ to obtain
	\[
	\Phi\T^c(s\varphi_{A\otimes B}) = \big(\T^c(\rho_A\Psi)\otimes\T^c(s\varphi_B)\big)\Phi\T^c(\manin_\Q)\ .
	\]
	Putting all of this together and writing $\tilde{\Delta}_{\lie^{\antishriek}}^\mathrm{bin}$ as a sum over all binary rooted trees, we obtain
	\begin{align*}
		\ell_n =&\ (\tilde{\gamma}_{\End_A}\otimes\vdw_C)(\T^c(\rho_A\Psi)\otimes\T^c(s\varphi_B))\Phi\T^c(\manin_\Q)\left(\sum_{\tau\in\brt_n}\tilde{\Delta}^\tau(\susp_n^{-1}\mu_n^\vee)\right)\\
		=&\ (\rho_A\Psi\tilde{\gamma}_\Q\otimes\vdw_C\T^c(s\varphi_B))\Phi\Bigg(\sum_{\tau\in\brt_n}T\Bigg(\tilde{\Delta}^\tau(\susp_n^{-1})\sum_{1\le i_1,\ldots,i_{n-1}\le k}\tau(f_{i_1}\otimes f_{i_1}^\vee,\ldots,f_{i_{n-1}}\otimes f_{i_{n-1}}^\vee)\Bigg)\Bigg)\\
		=&\ (\rho_A\Psi\otimes\vdw_C\T^c(s\varphi_B))(\bigstar)\ ,
	\end{align*}
	where
	\[
	(\bigstar) \coloneqq \sum_{\tau\in\brt_n}T\left(\tilde{\Delta}^\tau(\susp_n^{-1})\sum_{1\le i_1,\ldots,i_{n-1}\le k}(-1)^\epsilon\tilde{\gamma}_\Q(\tau(f_{i_1},\ldots,f_{i_{n-1}}))\otimes\tau(f_{i_1}^\vee,\ldots,f_{i_{n-1}}^\vee)\right)\ .
	\]
	The sign $(-1)^\epsilon$ comes from the map $\Phi$ doubling a tree and separating operations (as we have to switch some $f_i$'s and $f_i^\vee$'s). Now, noticing that that sign is exactly the one coming from dualization, we have
	\begin{align*}
		(\bigstar) =&\ \sum_{\tau\in\brt_n}T\left(\tilde{\Delta}^\tau(\susp_n^{-1})\sum_{1\le i_1,\ldots,i_{n-1}\le k}(-1)^\epsilon\tilde{\gamma}_\Q(\tau(f_{i_1},\ldots,f_{i_{n-1}}))\otimes\tau(f_{i_1}^\vee,\ldots,f_{i_{n-1}}^\vee)\right)\\
		=&\ \sum_{\tau\in\brt_n}T\left(\tilde{\Delta}^\tau(\susp_n^{-1})\sum_{1\le i_1,\ldots,i_{n-1}\le k}\tilde{\gamma}_\Q(\tau(f_{i_1},\ldots,f_{i_{n-1}}))\otimes\tau(f_{i_1},\ldots,f_{i_{n-1}})^\vee\right)\ .
	\end{align*}
	The equality of the two structures comes from the fact that
	\[
	\sum_{i=1}^{m(n)}q_i\otimes\tilde{\Delta}^\tau(q_i^\vee) = \sum_{1\le i_1,\ldots,i_{n-1}\le k}\tilde{\gamma}_\Q(\tau(f_{i_1},\ldots,f_{i_{n-1}}))\otimes\tau(f_{i_1},\ldots,f_{i_n})^\vee
	\]
	for all $\tau\in\brt_n$. This can be seen by writing the composition map $\gamma_\Q$ on $\T(F)^\tau$ in two ways. The formul{\ae} for the $\infty$-morphisms extending the maps in the retraction follow immediately from the explicit expressions for them given in \cite[Sect. 10.3.5--6]{vallette12} and the fact that the maps in the retraction act by the identity on $A$.
\end{proof}

\subsection{The non-symmetric case}

In the same situation as above, but with non-symmetric operads, the results take the following form.

\begin{theorem}
	The two $\A_\infty$-algebra structures on $A\otimes C$ obtained by applying the Homotopy Transfer Theorem in two different ways are equal, and the respective $\infty$-morphisms extending the morphisms of the retraction are related by
	\[
	(1\otimes i)_\infty = 1\otimes^\Psi i_\infty\qquad\text{and}\qquad(1\otimes p)_\infty = 1\otimes^\Psi p_\infty\ .
	\]
\end{theorem}

\section{Maurer--Cartan elements} \label{sect:MCel}

Theorem \ref{thm:mainThmDual} gives us a way to endow the hom space of certain coalgebras and algebras with an $\L_\infty$-algebra structure. It is then a natural question to ask what the Maurer--Cartan elements of the resulting $\L_\infty$-algebra represent. The interest into this question is given by the following ``philosophical'' principle attributed to Deligne and others, and made into a formal theorem using $\infty$-categories by Pridham \cite{pridham10} and Lurie \cite{lurie14}:

\medskip

\begin{center}
	\begin{minipage}{0.8\textwidth}
		\begin{center}
			``Every deformation problem in characteristic $0$ is encoded into the Maurer--Cartan elements of an $\L_\infty$-algebra.''
		\end{center}
	\end{minipage}
\end{center}

\medskip

In this section, we interpret the set of Maurer--Cartan elements of the homotopy Lie algebra obtained through the main theorem as the set of morphism between certain algebras. This will allow us to construct the deformation complex for morphisms of algebras over an operad in the next section. The dual setting, which will allow us to represent the Deligne--Hinich--Getzler $\infty$-groupoid for dg Lie algebras, is also treated.

\subsection{Maurer--Cartan elements of $\L_\infty$-algebras}

Let $(\g,d,\{\ell_n\}_{n\ge2})$ be an $\L_\infty$-algebra. Then, whenever it is defined, the \emph{Maurer--Cartan equation} is
\[
dx + \sum_{n\ge2}\frac{1}{n!}\ell_n(x,\ldots,x) = 0
\]
for $|x| = -1$. In particular, if $\g$ is actually a dg Lie algebra (i.e. $\ell_n = 0$ for all $n\ge 3$), then we recover the usual Maurer--Cartan equation
\[
dx + \frac{1}{2}[x,x] = 0\ .
\]

\begin{definition}
	The set of Maurer--Cartan elements of an $\L_\infty$-algebra $\g$ is denoted by $\MC(\g)$.
\end{definition}

\subsection{Twisting morphisms and bar--cobar constructions}

This part of the section is a slight variation of the material that can be found in \cite[Sect. 11.1--3]{vallette12}, where we consider general coalgebras instead of just conilpotent ones, which forces us to work with the complete algebras of Appendix \ref{sect:appendixCompletePalg} instead of general ones. Notice that we need the fact that the dg operad $\P$ is reduced in order to have a well-defined notion of complete $\P$-algebras. The various results given in the book \cite{vallette12} keep holding in this situation with only slight changes. For example, a morphism from a free complete $\P$-algebra to a \emph{complete} $\P$-algebra is completely determined by the image of the generators.

\medskip

Let $\C$ be a dg cooperad and let $\P$ be a dg operad. Suppose $\alpha:\C\to\P$ is a twisting morphism. Let $A$ be a complete $\P$-algebra and let $D$ be a $\C$-coalgebra. Then a \emph{twisting morphism with respect to $\alpha$} is a linear map $\varphi:D\to A$ of degree $0$ satisfying the ``Maurer--Cartan equation''
\[
\partial(\varphi) + \star_\alpha(\varphi) = 0\ ,
\]
where $\star_\alpha$ is the unary operator of degree $0$ defined by
\[
\star_\alpha(\varphi) \coloneqq \left(D\xrightarrow{\Delta_D}\C\hat{\circ} D\xrightarrow{\alpha\circ\varphi}\P\hat{\circ} A\xrightarrow{\gamma_A}A\right)\ .
\]
Notice that there is a passage from invariants to coinvariants which is left implicit. It is actuated using the isomorphism described in the introduction. We denote the set of solutions of this equation by $\Tw_\alpha(D,A)$.

\medskip

We can use $\alpha\in\Tw(\C,\P)$ to construct a functor from $\C$-coalgebras to complete $\P$-algebras, which we call the \emph{complete cobar construction} $\widehat{\Omega}_\alpha$ for obvious reasons. It takes a $\C$-coalgebra $D$ and gives back the complete $\P$-algebra
\[
\widehat{\Omega}_\alpha D\coloneqq\big(\widehat{\P}(D),d\coloneqq d_1+d_2\big)\ ,
\]
where $-d_1$ is the unique derivation extending the differentials of $\P$ and $D$, and $-d_2$ is the unique derivation extending the degree $-1$ map
\[
D\xrightarrow{\Delta_D}\C\hat{\circ}D\xrightarrow{\alpha\hat{\circ}\id_D}\P\hat{\circ}D = \widehat{\P}(D)\ .
\]

\begin{proposition}
	There is a natural bijection
	\[
	\hom_{\mathsf{dg}\widehat{\P}\text{-}\mathsf{alg}}(\widehat{\Omega}_\alpha D,A)\cong\Tw_\alpha(D,A)\ .
	\]
\end{proposition}

\begin{proof}
	The proof given in the book \cite[Prop. 11.3.1]{vallette12} goes through basically unchanged. Notice that it is crucial that $A$ is complete, since for example we need the fact that a morphism $\widehat{\Omega}_\alpha(D)\to A$ is equivalent to its restriction to $D$, and if $A$ is not complete we can only say that such a morphism gives a map $D\to A$, but not go in the other direction.
\end{proof}

Recall that there is a canonical twisting morphism $\pi:\bar\Q\to\Q$ associated to the counit of the bar--cobar adjunction, see \cite[Sect. 6.5.4]{vallette12}. Given a morphism $\Psi:\Q\to\P$ of dg operads, we can pull it back by $\pi$ to obtain the twisting morphism
\[
\psi \coloneqq \pi^*\Psi:\bar\Q\longrightarrow\P\ ,
\]
and thus the complete cobar functor
\[
\widehat{\Omega}_\psi:\mathsf{dg}\bar\Q\text{-}\mathsf{cog}\longrightarrow\mathsf{dg}\widehat{\P}\text{-}\mathsf{alg}\ .
\]

\subsection{Maurer--Cartan elements}

Let $\Psi:\Q\to\P$ be a morphism of dg operads, let $D$ be a $\bar(\susp\otimes\Q)$-coalgebra, and let $(A,F_\bullet A)$ be a filtered $\P$-algebra as defined in Appendix \ref{sect:appendixCompletePalg}. Then $s^{-1}D$ is canonically a $\bar\Q$-coalgebra. We can then compare the Maurer--Cartan equation in $\hom^\Psi(D,A)$ with the equation defining $\Tw_\psi$. We have a natural bijection
\[
\hom(D,A)\longrightarrow\hom(s^{-1}D,A)
\]
given by sending $\varphi\in\hom(D,A)$ to the linear map $s\varphi\in\hom(s^{-1}D,A)$ defined by
\[
s\varphi(s^{-1}x)\coloneqq (-1)^{|\varphi| + 1}\varphi(x)
\]
for any $x\in D$. This assignment anticommutes with the differentials.

\begin{theorem}
	Let $\varphi\in\hom(D,A)$ be of degree $-1$. Then we have
	\[
	\partial(s\varphi)+\star_\psi(s\varphi) = 0\qquad\Leftrightarrow\qquad d(\varphi) + \sum_{n\ge2}\frac{1}{n!}\ell_n(\varphi,\ldots,\varphi) = 0\ .
	\]
	In other words, there is a natural bijection
	\[
	\Tw_\psi(D,A)\cong\MC(\hom^\Psi(s^{-1}D,A))\ .
	\]
\end{theorem}

\begin{remark}
	Notice that, since $A$ is filtered, we can define a descending filtration by
	\[
	F_n\hom^\Psi(D,A)\coloneqq\hom(D,F_nA)\subseteq\hom(D,A)\ ,
	\]
	which makes $\hom^\Psi(D,A)$ into a filtered $\L_\infty$-algebra. Then, by Proposition \ref{prop:filtered implies complete} we have that $\hom^\Psi(D,A)$ is complete, so that it makes sense to speak of its Maurer--Cartan elements.
	
	\medskip
	
	Instead of taking $A$ filtered, one can also consider arbitrary $\P$-algebras but require the coalgebra $D$ to be conilpotent. This is equivalent to the coradical filtration $F^n_\C D$ of $D$ being exhaustive, see \cite[Sect. 5.8.4]{vallette12}. Then the $\L_\infty$-algebra $\hom^\Psi(D,A)$ is again filtered by
	\[
	F_n\hom^\Psi(D,A)\coloneqq\{f\in\hom(D,A)\mid F^n_\C D\subseteq\ker(f)\}\ ,
	\]
	and thus complete.
	
	\medskip
	
	We will find ourselves in the first situation when speaking of the Deligne--Hinich--Getzler $\infty$-groupoid in the application presented in Section \ref{subsection:cosimplicial}, while the second situation will make its appearance when defining the deformation complex of morphisms of $\P$-algebras in Section \ref{subsection:defCplx}.
\end{remark}

\begin{proof}
	As already remarked above, we have
	\[
	s(d\varphi) = -\partial(s\varphi)\ .
	\]
	Fix $n\ge2$, we will compare $\ell_n(\varphi,\ldots,\varphi)$ with $\star_\psi^{(n)}(s\varphi)$, the part of $\star_\psi(s\varphi)$ passing through
	\[
	\bar\Q(n)\otimes_{\S_n} D^{\otimes n}\subset\bar\Q\hat{\circ} D\ .
	\]
	We start by computing
	\begin{align*}
		\frac{1}{n!}\ell_n(\varphi,\ldots,\varphi) =&\ \frac{1}{n!}\gamma_{\hom(D,A)}\big(\genmaninDual_\Psi(\ell_n)\otimes\varphi^{\otimes n}\big)\\
		=&\ \gamma_A\big(\genmaninDual_\Psi(\ell_n)\circ\varphi\big)\Delta_D(n)\\
		=&\ \gamma_A\big(\big((s^{-1}\susp_n^{-1}\Psi(n))\mathrm{proj}^{(1)}\big)\circ\varphi\big)\Delta_D(n)\\
		=&\ \gamma_A\big((s^{-1}\susp_n^{-1}\Psi(n))\circ\varphi\big)\overline{\Delta}_D(n)
	\end{align*}
	Starting from the second line we left implicit the passage from invariants to coinvariants given by the usual isomorphism. The factor $n!$ come from the fact that in the definition of the algebraic structure on $\hom(D,A)$ we used another identification of invariants with coinvariants. In the last line, $\overline{\Delta}_D(n)$ denotes $(\mathrm{proj}^{(1)}\circ1)\Delta_D(n)$, the part of $\Delta_D$ living in $s\susp_n\Q(n)\otimes D^{\otimes n}$. To compute $\star_\psi^{(n)}(s\varphi)$, notice that the canonical twisting morphism $\pi:\bar\Q\to\Q$ is nothing else than the projection onto the weight $1$ part $s\Q\subset\bar\Q$, so that we have
	\begin{align*}
		\star_\psi^{(n)}(s\varphi) =&\ \gamma_A(\psi\circ s\varphi)\Delta_{s^{-1}D}(n)\\
		=&\ \gamma_A(\Psi \pi\circ s\varphi)\Delta_{s^{-1}D}(n)\\
		=&\ \gamma_A(\Psi\circ s\varphi)s^{-1}\overline{\Delta}_{s^{-1}D}(n)\ ,
	\end{align*}
	where $\overline{\Delta}_{s^{-1}D}(n)$ is a notation for $(\pi\circ1)\Delta_{s^{-1}D}$, the part of $\Delta_{s^{-1}D}$ contained in $s\Q(n)\otimes(s^{-1}D)^{\otimes n}$. To finish the proof, we compare their actions on an element of $s^{-1}D$. Let $x\in D$, we use Sweedler's notation and write
	\[
	\overline{\Delta}_D(n)(x) = s\susp_nq_{(0)}\otimes x_{(1)}\otimes\cdots\otimes x_{(n)}\ .
	\]
	We have
	\begin{align*}
		s\ell_n(\varphi,\ldots,\varphi)(s^{-1}x) =&\ -\ell_n(\varphi,\ldots,\varphi)(x)\\
		=&\ -\gamma_A\big((s^{-1}\susp_n^{-1}\Psi(n))\circ\varphi\big)\overline{\Delta}_D(n)(x)\\
		=&\ -\gamma_A\big((s^{-1}\susp_n^{-1}\Psi(n))\circ\varphi\big)(s\susp_nq_{(0)}\otimes x_{(1)}\otimes\cdots\otimes x_{(n)})\\
		=&\ (-1)^{n(n+|q_{(0)}|) + \epsilon + 1}\gamma_A\big((s^{-1}\susp_n^{-1}\Psi(n)(s\susp_nq_{(0)}))\otimes\varphi(x_{(1)})\otimes\cdots\otimes\varphi(x_{(n)})\big)\\
		=&\ (-1)^{n|q_{(0)}| + \epsilon + \frac{n(n-1)}{2}}\gamma_A\big(\Psi(q_{(0)})\otimes\varphi(x_{(1)})\otimes\cdots\otimes\varphi(x_{(n)})\big)\ .
	\end{align*}
	Notice that the sign in the first line is negative as $|\ell_n(\varphi,\ldots,\varphi)| = -1$. In the last two lines, the sign $\epsilon$ is such that
	\[
	(\varphi\otimes\cdots\otimes\varphi)(x_{(1)}\otimes\cdots\otimes x_{(n)}) = (-1)^\epsilon\varphi(x_{(1)})\otimes\cdots\otimes\varphi(x_{(n)})\ .
	\]
	For the other side of the equation, we have
	\[
	\overline{\Delta}_{s^{-1}D}(n)(s^{-1}x) = (-1)^{n|q_{(0)}| + \epsilon + 1 + \frac{n(n-1)}{2}}sq_{(0)}\otimes s^{-1}x_{(1)}\otimes\cdots\otimes s^{-1}x_{(n)}\ .
	\]
	Therefore, we have
	\begin{align*}
		\star_\psi^{(n)}(s\varphi)(s^{-1}x) =&\ \gamma_A(\Psi\circ s\varphi)s^{-1}\overline{\Delta}_{s^{-1}D}(n)(s^{-1}x)\\
		=&\ (-1)^{n|q_{(0)}| + \epsilon + 1 + \frac{n(n-1)}{2}}\gamma_A(\Psi\circ s\varphi)(q_{(0)}\otimes s^{-1}x_{(1)}\otimes\cdots\otimes s^{-1}x_{(n)})\\
		=&\ (-1)^{n|q_{(0)}| + \epsilon + \frac{n(n-1)}{2} + 1}\gamma_A\big(\Psi(q_{(0)})\otimes\varphi(x_{(1)})\otimes\cdots\otimes\varphi(x_{(n)})\big)\ .
	\end{align*}
	The two terms differ by a sign, concluding the proof.
\end{proof}

\begin{corollary} \label{cor:MCofHom}
	Let $A$ be a $\P$-algebra, and let $D$ be a $\bar(\susp\otimes\Q)$-coalgebra. If $A$ is complete, we have a natural bijection
	\[
	\hom_{\mathsf{dg}\P\text{-}\mathsf{alg}}(\widehat{\Omega}_\psi(s^{-1}D),A)\cong\MC(\hom^\Psi(D,A))\ .
	\]
	If $D$ is conilpotent, we have natural bijections
	\[
	\hom_{\mathsf{dg}\P\text{-}\mathsf{alg}}(\Omega_\psi(s^{-1}D),A)\cong\MC(\hom^\Psi(D,A))\cong\hom_{\mathsf{dg}\bar\Q\text{-}\mathsf{cog}}(s^{-1}D,\bar_\psi A)\ .
	\]
\end{corollary}

\subsection{The dual case}

In the dual setting, let $\Psi:\Q\to\P$ be a morphism of dg operads such that $\Q$ is finite dimensional in every arity, let $A$ be a complete $\P$-algebra, and let $C$ be a finite dimensional $\Omega((\susp^{-1})^c\otimes\Q)$-algebra. Then the dual $C^\vee$ of $C$ is naturally a $\bar(\susp\otimes\Q)$-coalgebra. Corollary \ref{cor:MCofHom} translates as follows in this context.

\begin{corollary} \label{cor:MCel}
	Let $A$ be a complete $\P$-algebra and let $C$ be a finite dimensional $\Omega((\susp^{-1})^c\otimes\Q^\vee)$-algebra. We have a natural bijection
	\[
	\hom_{\mathsf{dg}\P\text{-}\mathsf{alg}}(\widehat{\Omega}_\psi(s^{-1}C^\vee),A)\cong\MC(A\otimes^\Psi C)\ .
	\]
\end{corollary}

\subsection{The non-symmetric case}

There is a notion of Maurer--Cartan elements in an $\A_\infty$-algebra, where the relevant equation is
\[
dx + \sum_{n\ge2}m_n(x,\ldots,x) = 0\ .
\]
We denote again the set of all such elements by $\MC$. Notice that we do not have any factor $\tfrac{1}{n!}$ appearing. This is related to the fact that we never need to identify invariants and coinvariants in the ns setting, as we are not considering any group action. It follows that all our results hold over any field. The same reasoning as above gives the following statement.

\begin{theorem}
	Let $A$ be a $\P$-algebra and let $D$ be a $\hom(\bar(\susp\otimes\Q),\P)$-algebra. If $A$ is complete, we have a natural bijection
	\[
	\hom_{\mathsf{dg}\P\text{-}\mathsf{alg}}(\widehat{\Omega}_\psi(s^{-1}D),A)\cong\MC(\hom^\Psi(D,A))\ .
	\]
	If $D$ is conilpotent, we have natural bijections
	\[
	\hom_{\mathsf{dg}\P\text{-}\mathsf{alg}}(\Omega_\psi(s^{-1}D),A)\cong\MC(\hom^\Psi(D,A))\cong\hom_{\mathsf{dg}\bar\Q\text{-}\mathsf{cog}}(s^{-1}D,\bar_\psi(A))\ .
	\]
	Here the Maurer--Cartan elements are now taken in the $\A_\infty$-algebra $\hom^\Psi(D,A)$. Dually, under the usual assumption that $\Q$ is finite dimensional in every arity, if $C$ is a finite dimensional $\Omega((\susp^{-1})^c\otimes\Q^\vee)$-algebra and $A$ is complete, then we have a natural bijection
	\[
	\hom_{\mathsf{dg}\P\text{-}\mathsf{alg}}(\widehat{\Omega}_\psi(s^{-1}C^\vee),A)\cong\MC(A\otimes^\Psi C)\ .
	\]
\end{theorem}

\section{Applications to deformation theory}\label{sect:defTheory}

In this section, we apply the theorems we just proved to solve two problems in deformation theory. The first application is the construction of the deformation complex for morphisms of algebras over an operad. The second one is the representation of the Deligne--Hinich--Getzler $\infty$-groupoid for dg Lie algebras.

\subsection{First application: the deformation complex for morphisms of algebras} \label{subsection:defCplx}

Given an augmented dg operad $\P$ and two $\P$-algebras $X$ and $Y$, one might ask what is the deformation complex coding morphisms of $\P$-algebras from $X$ to $Y$. We solve this problem using Corollary \ref{cor:MCofHom}.

\medskip

Following Quillen \cite{quillen70}, the first step of the construction is to replace $X$ with a cofibrant replacement in the category of $\P$-algebras. In order to do so, one considers the canonical Koszul morphism
\[
\pi:\bar\P\longrightarrow\P
\]
and the associated bar-cobar adjunction between $\P$-algebras and $\bar\P$-coalgebras. One can then take $\Omega_\pi\bar_\pi X$ as a functorial resolution for $X$ and try to find the deformation complex for morphisms
\[
\Omega_\pi\bar_\pi X\longrightarrow Y\ .
\]
Notice that $\bar_\pi X$ is a conilpotent $\bar\P$-coalgebra, so that we can use the results proved in Section \ref{sect:MCel}.

\begin{proposition}
	We have
	\[
	\hom_{\mathsf{dg}\P\text{-}\mathsf{alg}}(\Omega_\pi\bar_\pi X,Y)\cong\MC\left(\hom^\P(s\bar_\pi X,Y)\right)\ ,
	\]
	where $\hom^\P\coloneqq\hom^{\id_\P}$.
\end{proposition}

\begin{proof}
	This is direct consequence of Corollary \ref{cor:MCofHom}.
\end{proof}

This motivates the following definition.

\begin{definition}\label{def:DefCplx}
	Let $\P$ be a dg operad and let $X,Y$ be two $\P$-algebras. The \emph{deformation complex of morphisms of $\P$-algebras} from $X$ to $Y$ is the $\L_\infty$-algebra $\hom^\P(s\bar_\pi X,Y)$.
\end{definition}

\begin{remark}
	In certain cases, one can find alternative versions for this object. For example, if $\P$ is binary Koszul, then one can give the structure of a Lie algebra to $\hom(s\bar_{\kappa}X,Y)$, where
	\[
	\kappa:\P^{\antishriek}\longrightarrow\P
	\]
	is the canonical twisting morphism, see \cite[Sect. 7.4.1]{vallette12}. The Maurer--Cartan elements of this Lie algebra correspond to the morphisms of $\P$-algebras from $\Omega_{\kappa}\bar_{\kappa}X$ to $Y$. The relation between this Lie algebra and the $\L_\infty$-algebra introduced in Definition \ref{def:DefCplx} is given by (the pre-dual version of) Proposition \ref{prop:ManinGenmanin}. Namely, we have the commutative square
	\begin{center}
		\begin{tikzpicture}
			\node (a) at (0,0){$\L_\infty$};
			\node (b) at (3.5,0){$\hom(\bar(\susp\otimes\P),\P)$};
			\node (c) at (0,-1.5){$\lie$};
			\node (d) at (3.5,-1.5){$\hom(\susp\otimes\P^{\antishriek},\P)$};
			
			\draw[->] (a)--node[above]{$\genmaninDual_\P$}(b);
			\draw[->] (a)--(c);
			\draw[->] (b)--node[right]{$f_\kappa^*$} (d);
			\draw[->] (c)--(d);
		\end{tikzpicture}
	\end{center}
	where the map $f_\kappa:\P^{\antishriek}\to\bar\P$ is the quasi-isomorphism of \cite[Prop. 6.5.8]{vallette12}. We can make $\bar_\kappa X$ into a $\bar\P$-coalgebra by pushing forward its structure along $f_\kappa$, and the resulting coalgebra is quasi-isomorphic to $\bar_\pi X$. Then the two $\L_\infty$-algebras $\hom(s\bar_{\kappa}X,Y)$ and $\hom^\P(s\bar_\pi X,Y)$ are quasi-isomorphic, and the quasi-isomorphism is filtered (e.g. with respect to the filtration induced by the coradical filtration of the coalgebras). Thus, we can apply the Dolgushev--Rogers theorem \cite[Thm. 2.2]{dolgushev15} to show that the deformation problems associated to the two algebras are the same.
\end{remark}

\subsection{Second application: a cosimplicial dg Lie algebra modeling Maurer--Cartan elements} \label{subsection:cosimplicial}

Here, we present an overview of how one can use the results given here to associate to every complete dg Lie algebra a new, small $\infty$-groupoid (i.e. Kan complex) representing its Maurer--Cartan elements which is equivalent to the well-known Deligne--Hinich--Getzler $\infty$-groupoid. The construction is functorial and the resulting functor from complete dg Lie algebra to simplicial sets is represented by a cosimplicial dg Lie algebra. These results are presented in full details (and more generality) in the article \cite{rn17cosimplicial}, where various theorems given here (namely, Theorems \ref{thm:mainThm}, and \ref{thm:twoLinftyStrAreEqual}, as well as Corollary \ref{cor:MCel}) are used in a crucial manner.

\medskip

The present article was written with the following application in mind. Let $(\g,F_\bullet\g)$ be a filtered dg Lie algebra. The \emph{Deligne--Hinich--Getzler $\infty$-groupoid} is the Kan complex
\[
\MC_\bullet(\g)\coloneqq\varprojlim_n\MC(\g/F_n\g\otimes\Omega_\bullet)\ ,
\]
where $\Omega_\bullet$ is the Sullivan simplicial dg commutative algebra of the polynomial differential forms on the geometric simplices. There is a simplicial contraction due to Dupont of $\Omega_\bullet$ onto a sub-simplicial set $C_\bullet$ which is finite dimensional at every level, which is just the cellular cochain complex of the geometric simplices. Using this retraction, one can obtain a filtered $\L_\infty$-algebra structure on $\g\otimes C_\bullet$. Using methods from homotopical algebra, we can prove that there is a homotopy equivalence of simplicial sets
\[
\MC_\bullet(\g)\simeq\MC(\g\otimes C_\bullet)\ .
\]
Finally, using Corollary \ref{cor:MCel} on $\g\otimes C_\bullet$ for the morphism $\Psi$ being the identity of the operad $\lie$, we have that
\[
\MC(\g\otimes C_\bullet)\cong\hom_\mathsf{dgLie}(\widehat{\Omega}_\pi(s^{-1}C_\bullet^\vee),\g)\ ,
\]
which provides thus a new model $\mc_\bullet\coloneqq\widehat{\Omega}_\pi(s^{-1}C_\bullet^\vee)$ for the space of Maurer--Cartan elements of a dg Lie algebra. This leads to the following result.

\begin{theorem}[{\cite[Cor. 5.3]{rn17cosimplicial}}]
	Let $\g$ be a complete dg Lie algebra. There is a weak equivalence of simplicial sets
	\[
	\MC_\bullet(\g)\simeq\hom_\mathsf{dgLie}(\mc_\bullet,\g)\ .
	\]
	It is natural in $\g$.
\end{theorem}

This, plus a study of the properties of the new $\infty$-groupoid, is the content of the article \cite{rn17cosimplicial}.

\section{Examples} \label{sect:sect5}

In this section, we will study the $\L_\infty$-algebra structures obtained from our main theorem (\ref{thm:mainThm}) for some canonical morphisms between the three most often appearing operads: the three graces $\com$, $\lie$ and $\ass$. Namely, we will study the identities of these operads and the sequence of morphisms
\[
\lie\stackrel{a}{\longrightarrow}\ass\stackrel{u}{\longrightarrow}\com\ ,
\]
where the first morphism corresponds to the antisymmetrization of the multiplication of an associative algebra, and where the second one corresponds to forgetting that the multiplication of a commutative algebra is commutative to get an associative algebra.

\medskip

Many more examples of less common, but still very interesting operads, both in the symmetric and in the ns case, as well as various morphisms relating them, can be found in \cite[Sect. 13]{vallette12}.

\subsection{Notations} \label{subsect:notationsExamples}

We will denote by $b\in\lie(2)$ the generating operation of $\lie$, i.e. the Lie bracket. The operad $\ass$ is the symmetric version  the non-symmetric operad $\as$ coding associative algebras. It is given by $\ass(n) = \k[\S_n]$. We denote the canonical basis of $\ass(n)$ by $\{m_\sigma\}_{\sigma\in\S_n}$. The element $m_\sigma\in\ass(n)$ corresponds to the operation
\[
(a_1,\ldots,a_n)\longmapsto a_{\sigma^{-1}(1)}\cdots a_{\sigma^{-1}(n)}
\]
at the level of associative algebras. The action of the symmetric group is of course given by $(m_\sigma)^\tau = m_{\sigma\tau}$. As before, we denote by $\mu_n\in\com(n)$ the canonical element. The morphism $a:\lie\to\ass$ is given by sending $b$ to $m_\id - m_{(12)}$, and corresponds to antisymmetrization at the level of algebras. The morphism $u:\ass\to\com$ is given by sending both $m_\id$ and $m_{(12)}$ to $\mu_2$. For the homotopy counterparts of the operads mentioned above: as before we denote by $\ell_n\in\L_\infty(n)$ the element of $\L_\infty(n)$ corresponding to the $n$-ary bracket. We have $\ass^{\antishriek}\cong(\susp^{-1})^c\otimes\ass^\vee$, thus in each arity $n\ge2$ the operad $\ass_\infty$ has $n!$ generators
\[
\overline{m}_\sigma\coloneqq s^{-1}\susp^{-1}_nm_\sigma^\vee\in\ass_\infty\ ,\qquad \sigma\in\S_n\ .
\]
The action of the symmetric group on these generators is given by $(\overline{m}_\sigma)^\tau = (-1)^\tau\overline{m}_{\sigma\tau}$. Finally, the operad $\C_\infty$ coding homotopy commutative algebras the same thing as an $\A_\infty$-algebra that vanishes on the sum of all non-trivial shuffles, see \cite{kadeishvili88} or \cite[Sect. 13.1.8]{vallette12}.

\subsection{The identity $\com\longrightarrow\com$} This is the simplest example. The identity of $\com$ induces the morphism
\[
\genmanin_\com:\L_\infty\longrightarrow\com\otimes\L_\infty
\]
which sends the element $\ell_n$ to
\[
\mu_n\otimes s^{-1}\susp_n\mu_n^\vee = \mu_n\otimes\ell_n\ .
\]
Therefore, it is the canonical isomorphism
\[
\L_\infty\cong\com\otimes\L_\infty\ .
\]
If $A$ is a commutative algebra and $C$ is an $\L_\infty$-algebra, then the operations on $A\otimes^\com C$ are given by
\[
\ell_n(a_1\otimes c_1,\ldots,a_n\otimes c_n) = (-1)^\epsilon \mu_n(a_1,\ldots,a_n)\otimes\ell_n(c_1,\ldots,c_n)\ ,
\]
where $(-1)^\epsilon$ is the sign obtained by commuting the $a_i$'s and the $c_i$'s.

\subsection{The identity $\ass\longrightarrow\ass$} Since the operad $\ass$ satisfies $\ass^{\shriek} = \ass$, the induced morphism is
\[
\genmanin_\ass:\L_\infty\longrightarrow\ass\otimes\ass_\infty\ .
\]
It sends $\ell_n$ to
\[
\sum_{\sigma\in\S_n}m_\sigma\otimes s^{-1}\susp_n m_\sigma^\vee = \sum_{\sigma\in\S_n}m_\sigma\otimes \overline{m}_\sigma = \sum_{\sigma\in\S_n}(-1)^\sigma(m_\id\otimes\overline{m}_\id)\ .
\]
If $A$ is an associative algebra and $C$ is an $\ass_\infty$-algebra, then the $\L_\infty$ operations on $A\otimes^\ass C$ are given by
\[
\ell_n(a_1\otimes c_1,\ldots,a_n\otimes c_n) = \sum_{\sigma\in\S_n}(-1)^{\sigma+\epsilon} m_e(a_{\sigma^{-1}(1)},\ldots a_{\sigma^{-1}(n)})\otimes\overline{m}_e(c_{\sigma^{-1}(1)},\ldots c_{\sigma^{-1}(n)})\ ,
\]
where $\epsilon$ is the sign obtained by switching the $a_i$'s and the $c_i$'s, and correspond therefore to a kind of antisymmetrization of $\ass_\infty$.

\subsection{The identity $\lie\longrightarrow\lie$} The last identity we have to look at is the identity of the operad $\lie$. It gives rise to a morphism of dg operads
\[
\genmanin_\lie:\L_\infty\longrightarrow\lie\otimes\C_\infty\ .
\]
It is of more complicated description, but comparing formul{\ae} we see that it is the same structure that is used in a fundamental way in the article \cite[pp.19--20]{turchin15} on Hochschild--Pirashvili homology.
	
\subsection{The forgetful morphism $u:\ass\longrightarrow\com$} This morphism is given by sending
\[
m_\sigma\longmapsto \mu_n
\]
for all $\sigma\in\S_n$. The corresponding morphism
\[
\genmanin_u:\L_\infty\longrightarrow\com\otimes\ass_\infty\cong\ass_\infty
\]
is given by
\[
\genmanin_u(\ell_n) = \sum_{\sigma\in\S_n}\mu_n\otimes s^{-1}\susp_nm_\sigma^\vee = \sum_{\sigma\in\S_n}\mu_n\otimes\overline{m}_\sigma = \mu_n\otimes\sum_{\sigma\in\S_n}(-1)^\sigma(\overline{m}_e)^\sigma.
\]
Therefore, under the canonical identification $\com\otimes\ass_\infty\cong\ass_\infty$, it is the standard antisymmetrization of an $\ass_\infty$-algebra structure giving an $\L_\infty$-algebra structure.

\subsection{The antisymmetrization morphism $a:\lie\longrightarrow\ass$} The induced morphism is a morphism of dg operads
\[
\genmanin_a:\L_\infty\longrightarrow\ass\otimes\C_\infty\ .
\]
It can be interpreted as follows: a $\C_\infty$-algebra can be seen as an $\ass_\infty$-algebra vanishing on the sum of all non-trivial shuffles (see \cite{kadeishvili88}). That is, we have a natural morphism of dg operads
\[
i:\ass_\infty\longrightarrow\C_\infty\ ,
\]
which is in fact given by $\Omega((\susp^{-1})^c\otimes a^\vee)$. Now we can use the second part of Theorem \ref{thm:mainThm} telling us that
\[
\genmanin_a = \genmanin_{1_\ass a} = (1\otimes i)\genmanin_\ass\ .
\]
Therefore, the $\L_\infty$-algebra structure on the tensor product of an associative and a $\C_\infty$-algebra is given by first looking at the $\C_\infty$-algebra as an $\ass_\infty$-algebra, and then antisymmetrizing the resulting $(\ass\otimes\ass_\infty)$-algebra as already done above.

\subsection{The non-symmetric case} The analogues to the three graces in the non-symmetric setting are the operad $\as$ encoding associative algebras, which we already know well, the operad $\mathit{Dend}$ of \emph{dendriform algebras} (\cite[Sect. 13.6.5]{vallette12}), and the operad $\mathit{Dias}$ encoding \emph{diassociative algebras} (\cite[Sect. 13.6.7]{vallette12}). They fit into a sequence of morphisms
\[
\mathit{Dias}\longrightarrow\as\longrightarrow\mathit{Dend}\ .
\]
We get induced morphisms from $\A_\infty$ to
\[
\mathit{Dias}\otimes\mathit{Dend}_\infty\ ,\quad\A_\infty\ ,\quad\mathit{Dend}\otimes\mathit{Dias}_\infty\ ,\quad\mathit{Dend}_\infty\,\quad\text{and}\quad\mathit{Dend}\otimes\A_\infty\ .
\]
We leave their explicit computation to the interested reader.

\appendix
\section{Complete \texorpdfstring{$\P$}{P}-algebras} \label{sect:appendixCompletePalg}

Fix a reduced operad $\P$. We will give the definition of complete $\P$-algebras and filtered $\P$-algebras, and state some basic facts about them, skipping most proofs.

\medskip

To the operad $\P$ we can associate the endofunctor on the category of (unbounded) chain complexes
\[
\widehat{\P}:\ch\longrightarrow\ch
\]
given by
\[
\widehat{\P}(V) \coloneqq \prod_{n\ge1}\P(n)\otimes_{\S_n}V^{\otimes n}.
\]

\begin{lemma}
	The usual unit and composition of $\P$ induce the structure of a monad on $\widehat{\P}$.
\end{lemma}

\begin{proof}
	By inspection. Notice that, since $\P$ is reduced, every sum involved finite and allows us to switch products and tensor products.
\end{proof}

\begin{lemma}
	We have a canonical morphism of monads
	\[
	\P\longrightarrow\widehat{\P}\ .
	\]
	In particular, every complete $\P$-algebra is a $\P$-algebra.
\end{lemma}

\begin{proof}
	The morphism is induced by the natural inclusion of direct sums of chain complexes into products.
\end{proof}

\begin{definition}
	A \emph{complete $\P$-algebra} is an algebra over the monad $\widehat{\P}$, seen as a $\P$-algebra.
\end{definition}

\begin{theorem}
	The free complete $\P$-algebra over a chain complex $V$ is given by $\widehat{\P}(V)$. In other words, every morphism of $\widehat{\P}$-algebras with $\widehat{\P}(V)$ as domain is completely characterized by its restriction to $V$.
\end{theorem}

One can also consider filtered (i.e. topological) $\P$-algebras.

\begin{definition}
	A \emph{filtered $\P$-algebra} $(A,F_\bullet A)$ is a $\P$-algebra $A$ together with a descending filtration
	\[
	A=F_1A\supseteq F_2A\supseteq F_3A\supseteq\cdots
	\]
	satisfying the following properties:
	\begin{enumerate}
		\item The filtration is closed under the differential of $A$, i.e.
		\[
		d_A(F_nA)\subseteq F_nA
		\]
		for all $n\ge1$.
		\item The composition map respects the filtration, that is
		\[
		\gamma_A(\P(k)\otimes_{\S_k}F_{n_1}A\otimes\cdots\otimes F_{n_k})\subseteq F_{n_1+\cdots+n_k}A
		\]
		for any $k\ge1$ and $n_1,\ldots n_k\ge1$.
		\item The algebra $A$ is complete with respect to the filtration, that is to say that
		\[
		A\cong\varinjlim_nA/F_nA
		\]
		as $\P$-algebras via the natural morphism.
	\end{enumerate}
\end{definition}

\begin{proposition} \label{prop:filtered implies complete}
	If $(A,F_\bullet A)$ is a filtered $\P$-algebra, then $A$ is a complete $\P$-algebra.
\end{proposition}

\begin{proof}
	Suppose $(A,F_\bullet A)$ is filtered, then every $a\in A$ can be written as
	\[
	\{a_n\}_{n\ge1}\in\varinjlim_nA/F_nA\ ,
	\]
	where $a_n\in A/F_nA$ is sent to $a_m$ by the obvious projection whenever $n\ge m$. So let
	\[
	x\coloneqq\left(p_k\otimes\{a^{k,1}_{n_1}\}_{n_1}\otimes\cdots\otimes\{a^{k,k}_{n_k}\}_{n_k}\right)_{k\ge1}\in\widehat{\P}(A)\ .
	\]
	Then we define
	\[
	\widehat{\gamma}_A(x)\coloneqq\left\{\sum_{k<m}\gamma_{A/F_mA}(p_k\otimes a^{k,1}_m\otimes\cdots\otimes a^{k,k}_m)\mod F_mA\right\}_{m\ge1}\in\varinjlim_nA/F_nA\cong A\ .
	\]
	It is straightforward to check that this makes $A$ into a complete $\P$-algebra.
\end{proof}

Finally, the free complete $\P$-algebra is canonically a filtered $\P$-algebra.

\begin{proposition}
	Let $V$ be a chain complex. Then the filtration
	\[
	F_n\widehat{\P}(V)\coloneqq\prod_{k\ge n}\P(k)\otimes_{\S_k}V^{\otimes k}
	\]
	makes $\widehat{\P}(V)$ into a filtered $\P$-algebra.
\end{proposition}

\bibliographystyle{alpha}
\bibliography{Deformation_theory_with_homotopy_algebra_structures_on_tensor_products}

\end{document}